\documentclass[reqno]{amsart}
\usepackage{amssymb}
\usepackage[usenames, dvipsnames]{color}
\usepackage{hyperref}
\usepackage{marginnote}
\usepackage{todonotes}

\numberwithin{equation}{section}

\usepackage{verbatim}

\usepackage{esvect} 
\usepackage{esint} 
\usepackage{mathtools} 
\DeclarePairedDelimiter{\norm}{\lVert}{\rVert}

\newcommand\ep{\varepsilon}

\newcommand\zz{\mathbb{Z}}

\newcommand\rr{\mathbb{R}}
\newcommand\cc{\mathbb{C}}

\newcommand\hh{\mathbb H}
\newcommand\cmm{\mathcal{M}} 
\newcommand\cdd{\mathcal{D}} 
\newcommand\css{\mathcal{S}}
\newcommand\cbb{\mathcal{B}}

\newcommand\cff{\mathcal{F}}

\def\aa{\mathcal{A}}
\def\cC{\mathcal{C}}
\def\pf{(-\Delta)^{\sigma/2}}

\newcommand{\real}{\mathrm{Re}}
\newcommand{\imag}{\mathrm{Im}}

\DeclareMathOperator*{\supp}{supp}

\newcommand{\qx}[1]{\quad \text{#1 }}

\theoremstyle{plain}
\newtheorem{theorem}{Theorem}[section]
\newtheorem{lemma}[theorem]{Lemma}
\newtheorem{corollary}[theorem]{Corollary}
\newtheorem{proposition}[theorem]{Proposition}

\theoremstyle{definition}

\newtheorem{assumption}[theorem]{Assumption}

\theoremstyle{remark}
\newtheorem{remark}[theorem]{Remark}

\makeatletter
\@namedef{subjclassname@2020}{%
  \textup{2020} Mathematics Subject Classification}
\makeatother

\begin{document}
\title[Parabolic equations with space-time non-local operators]{Sobolev estimates for fractional parabolic equations with space-time non-local operators}

 \author[H. Dong]{Hongjie Dong}
 \address[H. Dong]{Division of Applied Mathematics, Brown University, 182 George Street, Providence, RI 02912, USA}

 \email{Hongjie\_Dong@brown.edu}
 \thanks{H. Dong was partially supported by the Simons Foundation, grant no. 709545, a Simons fellowship, grant no. 007638, and the NSF under agreement DMS-2055244.}

 \author[Y. Liu]{Yanze Liu}
 \address[Y. Liu]{Department of Mathematics, Brown University, 151 Thayer Street, Providence, RI 02912, USA}

 \email{Yanze\_Liu@brown.edu}


\begin{abstract}
We obtain $L_p$ estimates for {fractional parabolic} equations with space-time non-local operators
\begin{align*}
   \partial_t^\alpha u - Lu + \lambda u= f \qx{in} (0,T) \times \rr^d ,
\end{align*}
where {$\partial_t^\alpha u$ is the Caputo fractional derivative of order} $\alpha \in (0,1]$, {$T\in (0,\infty)$,} and
\begin{align*}
    Lu(t,x) := \int_{\rr^d} \bigg( u(t,x+y)-u(t,x) - y\cdot \nabla_xu(t,x)\chi^{(\sigma)}(y)\bigg)K(t,x,y)\,dy
\end{align*}
is an integro-differential operator in the spatial variables. Here we do not impose any regularity assumption on the kernel $K$ with respect to $t$ and $y$.
We also derive a weighted mixed-norm estimate for the equations with operators that are local in time, i.e., $\alpha = 1$, which extend the previous results in \cite{MP,wuha} by using a quite different method.
\end{abstract}

\maketitle

\tableofcontents

\section{Introduction}
This paper is devoted to a study of $L_p$ (and $L_{p,q}$) estimates for non-divergence form equations with space-time non-local operators of the form
\begin{align*}
   \partial_t^\alpha u - L u + \lambda u = f \qx{in} (0,T) \times \rr^d \stepcounter{equation}  \tag{\theequation}\label{nondiv}
\end{align*}
with the zero initial condition $u(0,\cdot) = 0 $, where {$T\in (0,\infty)$,} $\partial_t^\alpha u$ is the Caputo fractional derivative of order $\alpha \in (0,1]$, and $L$ is an integro-differential operator
\begin{align*}
    Lu(t,x) := \int_{\rr^d} \bigg( u(t,x+y)-u(t,x) - y\cdot \nabla_xu(t,x)\chi^{(\sigma)}(y)\bigg)K(t,x,y)\,dy, \stepcounter{equation} \tag{\theequation}\label{op}
\end{align*}
{where $\chi^{(\sigma)}(y)=1_{\sigma\in (1,2)}+1_{\sigma=1}1_{|y|\le 1}$, and $K$ is nonnegative.}
A simple example of $L$ is the fractional Laplacian $-\pf$ for $\sigma\in(0,2)$.
We refer the reader to Section \ref{2} for the precise definitions of $\partial_t^\alpha u$ and {$L$}.

The fractional parabolic equations of the form \eqref{nondiv} have applications in various fields, including physics and probability theory. For example, by viewing the fundamental solution of \eqref{nondiv} as the probability density of a peculiar self-similar stochastic process evolving in time, we can see that the equation is related to the continuous-time random walk \cite{rw}.

There are many works regarding $L_p$ ({or} $L_{p,q}$) estimates for parabolic equations with either fractional derivatives in time or integro-differential operators in space or both. Equations {with non-local time derivatives} of the form
\begin{align*}
   \partial_t^\alpha u - a^{ij}(t,x) D_{ij} u  = f(t,x),\stepcounter{equation}  \tag{\theequation}\label{fs}
\end{align*}
were studied in \cite{P1,P2,dong19,dong20,dong21,DP}. In \cite{P1}, the authors derived mixed-norm estimates for {\eqref{fs} with} $\alpha\in(0,2)$ under the assumption that $a^{ij}$ are piecewise continuous in $t$ and uniformly continuous in $x$. In \cite{P2}, a weighted mixed-norm estimate was obtained when $\alpha\in(0,2)$  and $a^{ij}{=\delta_{ij}}$. More general coefficients $a^{ij}$ were studied in \cite{dong19,dong20}, where the $L_p$ and the weighted mixed-norm estimates were obtained respectively for $\alpha \in(0,1)$ under the assumption that $a^{ij}$ have small mean oscillations with respect to the spatial variables. Quite recently, in \cite{dong21} the weighted mixed-norm estimates are obtained for $\alpha \in(1,2)$ under the assumption that $a^{ij}$ have small oscillations in $(t,x)$. Furthermore, a weighted mixed-norm regularity theory for \eqref{fs} was obtained in \cite{DP} {when the coefficients are at least uniformly continuous in both $t$ and $x$}.

Equations {with non-local operators in $x$} of the form
\begin{align*}
   \partial_t u(t,x) - {Lu(t,x)} = f(t,x) \stepcounter{equation}  \tag{\theequation}\label{fmss}
\end{align*}
with $L$ defined in \eqref{op} were studied in \cite{M92,MP,MC,wuha,dyk}. In \cite{M92}, using {the} Fourier multiplier theorem, the authors derived an $L_p$ estimate for \eqref{fmss} under the conditions that $p$ is sufficiently large, {and the operator satisfies an ellipticity condition,} and $K(t,x,y)|y|^{d+\sigma}$
is homogeneous of order zero and smooth in $y$, and its derivatives in $y$ are continuous in $x$ uniformly in $(t,y)$.
In \cite{MP}, an $L_p$ estimate for \eqref{fmss} was obtained under the conditions that $p$ is sufficiently large, and $K(t,x,y)|y|^{d+\sigma}$ is bounded from above and below and H\"older continuous in $x$ uniformly in $(t,y)$. {No regularity assumption is assumed for $K$ in $t$ and $y$.} In particular, when $K = K(t,y)$, i.e. independent of $x$, and $f$ is sufficiently smooth, using a probabilistic method, they constructed the solution for \eqref{fmss} explicitly as the expected value involving some stochastic processes, and when $K = K(t,x,y)$, they applied the frozen coefficient argument. {For time-independent equations wit $K=K(y)$, such result was obtained earlier in \cite{dong12} by using a purely analytic method. It has been an open question whether in the case when $K=K(t,x,y)$, the restriction of $p$ in \cite{MP} can be removed. This problem was resolved} in a recent paper \cite{dyk} by using a weighted estimate and an {extrapolation} argument. In \cite{wuha}, by a probabilistic representation of the solution, the authors investigated the $L_p$-maximal regularity of  \eqref{fmss} as well as more general equations with singular and non-symmetric L\'evy operators.

Sobolev type estimates for equations with space-time non-local operators similar to \eqref{nondiv} were studied in \cite{kim20}  by deriving fundamental solutions using probability density functions of some stochastic processes.
Particularly, the equation
\begin{align*}
   \partial_t^\alpha u - \phi(\Delta) u = f \stepcounter{equation}  \tag{\theequation}\label{fms}
\end{align*}
is considered, where $\alpha \in (0,1)$ and $\phi$ is a Bernstein function satisfying certain growth condition.
The authors obtained a mixed-norm estimate of \eqref{fms} based on estimates of the fundamental solution from \cite{333} together with the Calder\'on-Zygmund theorem.

In this paper, we obtain the a priori estimate and the unique solvability for \eqref{nondiv} with space-time non-local operators. Compared to \cite{kim20}, our estimates do not directly rely on the representation of the fundamental solution. However, to apply the method of continuity, we need the results in \cite{kim20} or \cite{LP}. Indeed, for $K= K(t,y)$, we only require that $K(t,y)|y|^{d+\sigma}$ is bounded from above and below {and measurable}, an assumption for which the fundamental solution is unavailable. Furthermore, when $\alpha=1$, we obtain the weighted mixed-norm estimate for \eqref{nondiv}, where the weights are taken in both spatial and time variables.
The results are further extended to the case when $K=K(t,x,y)$, under the assumption that $K$ is H\"older continuous in $x$ uniformly in $(t,y)$. See Assumptions \ref{aaa1} and \ref{aaa2} for details. We also consider equations which contain lower-order terms. These, in particular, extend the aforementioned results in \cite{MP,wuha,dyk} to the weighted mixed-norm setting by using a quite different method.
Our main theorem reads that for any $p\in (1, \infty)$ and the kernel $K$ such that $K(t,y)|y|^{d+\sigma}$ is bounded {from above and below}, if $u$ satisfies (\ref{nondiv}) with the zero initial condition, then we have
$$
 {\norm{ (1-\Delta)^{\sigma/2}u}_{L_p(0,T) \times \rr^d)}}  + \norm{ \partial_t^\alpha u}_{L_{p}((0,T) \times \rr^d)}
\le N \norm{f}_{L_{p}((0,T) \times \rr^d)},
$$
where $N$ is independent of $u$ and $f$. Furthermore, for any $f \in L_{p}((0,T) \times \rr^d)$, there exists a unique solution $u$ to (\ref{nondiv}) with the zero initial condition in the appropriate Sobolev spaces defined in Section \ref{2}. Moreover, when $\alpha=1$, the a priori estimate and the unique solvability hold for \eqref{fmss} in the appropriate weighted mixed-norm Sobolev spaces {with Muckenhoupt weights}.

For the proof of the $L_p$ estimates for \eqref{nondiv} when $\alpha\in(0,1]$, we apply a  level set argument together with Lemma \ref{a6} (``crawling of ink spots {lemma}") by adapting the argument in \cite{dong19}. More precisely, we first prove the theorem for $p=2$ by using the Fourier transform. Then, we apply a bootstrap argument: assuming the theorem holds for $p_0$ and estimating the solution by a decomposition, we show that the theorem holds for $p \in(p_0,p_1)$ for some $p_1 > p_0$ such that the increment $p_1-p_0$ is independent of $p_0$. Indeed, for $(t_0,x_0)\in(0,T]\times \rr^d$, we decompose the solution $u$ to \eqref{nondiv} into $u = w+v$ such that
\[\partial_t^\alpha w - L w + \lambda w= \zeta f \qx{in} (t_0-1,t_0)\times \rr^d\stepcounter{equation} \tag{\theequation}\label{itrr} \]
{with the zero initial condition at $t=t_0-1$,}
where $\zeta$ is a {suitable} cutoff function satisfying $\zeta = 1 $ in $Q_1(t_0,x_0)$.
Then $v$ satisfies the homogeneous equation in $Q_1(t_0,x_0)$.
We bound $w$ by the maximal function of $f$ using the global $L_{p_0}$ estimate and prove a higher integrability result for $v$. See Proposition \ref{eees}.
Compared to \cite{dong19}, our proof is more involved since we need to deal with the nonlocality in both space and time.
For the proof of the weighted mixed-norm estimates of \eqref{nondiv} when $\alpha = 1$, we derive a mean oscillation estimate of $\pf u$, where $u$ is a solution to \eqref{nondiv}{, by using an iteration argument}. To this end, for $t_0\in(0,T]$, we decompose $u = w + v$, where
$$
\partial_tw -Lw+ \lambda w =f\qx{in}(t_0-1,t_0)\times \rr^d
$$
{so that $v$ satisfies the homogeneous equation in the strip $(t_0-1,t_0)\times \rr^d$.}
Compared to the estimate of $w$ in the previous case (cf. \eqref{itrr}), {here we cannot directly apply the global $L_{p_0}$ estimate since the right-hand side is not compactly supported in $x$. Our idea is to} use a localization and iteration argument. See Proposition \ref{nonhomo}. We also establish a H\"older estimate of $v$ by using a delicate bootstrap argument: first derive a local estimate by using an estimate of the commutator (Lemma \ref{com}) and then apply the Sobolev embedding theorems (Lemmas \ref{a5} and \ref{aa7}). {For this proof, it is crucial that $v$ satisfies the homogeneous equation in the strip $(t_0-1,t_0)\times \rr^d$ instead of just in a cylinder.}
The weighted mixed-norm estimates then follow by using the Fefferman-Stein theorem for sharp functions and the Hardy-Littlewood theorem for (strong) maximal functions in weighted mixed-norm spaces. Finally, for the general cases when $K = K(t,x,y)$, we use a {boundedness} result in \cite{dyk} with a perturbation argument and an extrapolation theorem.

In the future, we plan to extend the results by considering the cases of non-zero initial data, equations in domains, and divergence form equations. Indeed, some {interesting} work have been done in those directions. For example, solutions for equations in domains were studied in \cite{BB} using the probabilistic method. Divergence form equations were also studied in \cite{MA,ME} with various assumptions on the kernel $K$.
More precisely, H\"older estimates and $L_p$ estimates were obtained for the divergence form equations in \cite{MA} and \cite{ME}, respectively. Furthermore, there are results about divergence form non-local elliptic equations in \cite{KU, S1,S2}, where more general operators were considered. For example, in \cite{S2}, the authors obtained some regularity results for operators with VMO kernels and non-linearity. {Another interesting question is whether} the weighted mixed-norm estimates {still hold} for \eqref{nondiv} when $\alpha \in (0,1)$, where the weights are taken over the time or the spatial variables. See Remarks \ref{657} and \ref{658} for more information.

The remaining part of the paper is organized as follows. In Section \ref{2}, we introduce notation, definitions, and the main results of the paper. In Section \ref{3}, we obtain the $L_p$ estimates using the level set argument. In Section \ref{4}, we first estimate the mean oscillation {of solutions} by using the decomposition mentioned above, and then prove the weighted mixed-norm estimates. In Appendix \ref{A}, we prove miscellaneous lemmas used in the main proofs.

\section{Notation and Main Results}\label{2}
We first introduce some notation used throughout the paper. For $\alpha \in (0,1)$ and $S \in \rr$, we denote
\[
I^\alpha_S u(t,x)
= \frac{1}{\Gamma(\alpha)}
\int_S^t (t - s)^{\alpha - 1} u (s,x) \, ds,\quad t\ge S,\, x\in \rr^d,
\]
and
\[
\partial_t^\alpha u(t,x)
= \frac{1}{\Gamma(1 - \alpha)}
\int_0^t (t-s)^{-\alpha} \partial_t u (s,x) \, ds,\quad t\ge 0,\, x\in \rr^d.
\]
It is easily seen that $\partial_t^\alpha u = \partial_t I_0^{1-\alpha} u$ for a sufficiently smooth $u$ with $u(0,x) = 0$.

We use $\cff(u)$ and $\hat{u}$ to denote the Fourier transform of $u$. If $u = u(t,x)$, then $\cff(u)(t,\xi)$ and $\hat{u}(t,\xi)$ denote the Fourier transform of $u$ in $x$ for a fixed time $t$.
For $p\in (0,\infty)$ and $\sigma \in(0,2)$, recall the definition of the Bessel potential space
$$H_{p}^\sigma(\rr^d) = \{ u \in L_p(\rr^d) : (1-\Delta)^{\sigma/2}u \in L_p(\rr^d)\}$$
and
$$\norm{u}_{H_{p}^\sigma(\rr^d)} := \norm{ (1-\Delta)^{\sigma/2}u}_{L_p(\rr^d)},$$
where
\begin{align*}
(1-\Delta)^{\sigma/2}u(x) = \cff^{-1}\Big[(1 + |\xi|^2)^{\sigma/2}\cff(u)(\xi)\Big](x).
\end{align*}

For $p,q\in(1,\infty)$ {and $-\infty<S<T<\infty$}, we define
$L_{p,q} \big((S,T )\times \rr^d\big)$ to be the set of all measurable functions $f$ defined on $(S,T) \times \rr^d$ satisfying
\[
\norm{f}_{L_{p,q} ((S,T ){\times \rr^d})} :=
\bigg( \int_S^T
\big( \int_{\rr^d}
| f(t,x)|^q \, dx
\big)^{p/q}
\, dt
\bigg)^{1/p}
< \infty.
\]
When $p=q$, we write $L_p \big((S,T )\times \rr^d\big) = L_{p,p}\big((S,T )\times \rr^d\big)$. Furthermore, for $\alpha \in (0,1]$ and $\sigma\in (0,2)$, we denote $ \hh_{p,q}^{\alpha, \sigma}(S,T)$ to be the collection of functions such that $u \in L_p\big((S,T);H_q^{\sigma}(\rr^d)\big)$, $\partial_t^\alpha u \in L_{p,q}\big((S,T)\times \rr^d\big)$, and
$$\norm{u}_{\hh_{p,q}^{\alpha, \sigma}(S,T)} :=\bigg( \int_S^T
\big( \norm{u(t,\cdot)}_{H_q^\sigma(\rr^d)}
\big)^{p/q}
\, dt
\bigg)^{1/p}+ \norm{ \partial_t^\alpha u}_{L_{p,q} ((S,T)\times \rr^d)}.$$
We write $u \in \hh_{p,q,0}^{\alpha, \sigma} ( S,T) := \hh_{p,q,0}^{\alpha, \sigma} \big(( S,T) \times \rr^d \big)$
if there exists a sequence of functions $\{ u_n \}$ such that  $u_n \in C^\infty ( [S,T] \times \rr^d )$ with $u_n(S,x) = 0$ vanish for large $|x|$, and
\[
\norm{u_n - u}_{\hh_{p,q}^{\alpha, \sigma} (S,T )}
\to 0  \qx{as} n \to \infty.
\]
Moreover, for a domain $\Omega \subset \rr^d$ and $ u$ defined on $(S,T)\times \Omega$, we write $u \in \hh_{p,q,0}^{\alpha, \sigma} \big(( S,T) \times \Omega \big)$ if
there exists a extension of $u$ to $(S,T)\times \rr^d$, i.e.,
\begin{align*}
  \norm{ u}_{\hh_{p,q}^{\alpha, \sigma} ((S,T )\times \Omega)} := \inf \{ \norm{\overline{u}}_{\hh_{p,q}^{\alpha, \sigma}({S,}T)} : \overline{u} \in \hh_{p,q,0}^{\alpha, \sigma}({S,}T) \qx{and }\overline{u}|_{{(S,T)\times}\Omega} = u  \} < \infty .
\end{align*}
We denote
$\hh_{p,0}^{\alpha, \sigma}( (S,T )\times \Omega ):=\hh_{p,p,0}^{\alpha, \sigma} ((S,T )\times \Omega )$ if $p=q$. Furthermore, we take $\psi(x) = 1/(1 + |x|^{d+ \sigma})$, and {denote}
\[\norm{u}_{L_p((0,T);L_1(\rr^d,\psi))} :=\norm{\psi u}_{L_p((0,T);L_1(\rr^d))}. \stepcounter{equation} \tag{\theequation}\label{itr}
\]

{For $T\in (0,\infty)$, we denote $(0,T) \times \rr^d :=\rr^d_T $ and}
we often denote $\hh_{p,q}^{\alpha, \sigma}(T) := \hh_{p,q}^{\alpha, \sigma}(0,T)$ and $\hh_{p,q,0}^{\alpha, \sigma}(T):= \hh_{p,q,0}^{\alpha, \sigma}(0,T)$.
We use the notation $u \in  \hh_{p,0,\mathrm{loc}}^{\alpha, \sigma} ( \rr^d_T )$ to indicate a function satisfying $u \in  \hh_{p,0}^{\alpha, \sigma} ( (0,T) \times B_R ) $ for all $R>0$.

For $\alpha \in (0,1)$, $\sigma\in (0,2)$, $r_1,r_2 >0$, and $(t,x)\in \rr^{d+1}$, we denote the parabolic cylinder by
\[
Q_{r_1, r_2} (t,x) = (t-r_1^{\sigma/\alpha}, t) \times B_{r_2} (x)
\quad \text{and}\quad
Q_r(t,x) = Q_{r,r}(t,x),
\]
where $B_{r_2} (x) = \{ y \in \rr^d: |y - x| < r_2 \}$.
We write $B_r$ and $Q_r$ for $B_r(0)$ and $Q_r(0,0)$. Furthermore, for $f\in L_{1,\text{loc}}$ defined on $\cdd \subset \rr^{d+1}$ and $(t,x) \in \cdd$, we {define its} maximal function and strong maximal function, respectively, by
\[
\cmm f (t,x) = \sup_{Q_r (s,y) \ni (t,x)}
\fint_{Q_r (s,y)}
|f(r,z) | \chi_\cdd \, dz \, dr
\]
and
\[
(\css \cmm f ) (t,x)= \sup_{Q_{r_1, r_2} (s,y) \ni (t,x)}
\fint_{Q_{r_1, r_2} (s,y)}
|f(r,z) | \chi_\cdd \, dz \, dr.
\]

Next, for $p \in (1, \infty)$, $k \in \{ 1,2,\ldots \}$, let $A_p (\rr^k, dx)$ be the set of all non-negative functions $w$ on $\rr^k$ such that
\[
[w]_{A_p(\rr^k)}
:= \sup_{x_0 \in \rr^k , r> 0}
\biggl(
\fint_{B_r(x_0) } w(x) \, dx
\biggr) \biggl(
\fint_{B_r(x_0) } (w(x))^{-\frac{1}{p-1}} \, dx
\biggr)^{p-1}
< \infty,
\]
where $B_r(x_0) = \{ x \in \rr^k: |x - x_0| < r \}$. Furthermore, for a constant $M_1 > 0$, we write $[w]_{p,q} \le M_1$ if $w = w_1(t) w_2(x)$ for some $w_1$ and $w_2$ satisfying
\[
w_1(t) \in A_p (\rr, dt),
\quad
w_2(x) \in A_q(\rr^d, dx),
\qx{and}
[w_1]_{A_p(\rr)},[w_2]_{A_q(\rr^k)} \leq M_1.
\]
We denote $L_{p,q,w} ( \rr^d_T )$ to be the set of all measurable functions $f$ defined on $\rr^d_T$ satisfying
\[
\norm{f}_{L_{p,q,w} ( \rr^d_T )}:=
\bigg( \int_0^T
\big( \int_{{\rr^d}}
| f(t,x)|^q w_2(x) \, dx
\big)^{p/q}
w_1(t) \, dt
\bigg)^{1/p}
< \infty.
\]
When $p=q$ and $w\equiv 1$, $L_{p,q,w} ( \rr^d_T )$ becomes the usual Lebesgue space $L_p ( \rr^d_T )$.

We write $N = N (\cdots)$ if the constant $N$ depends only on the parameters in the parentheses.

Next, we present the assumptions for the operator. In this paper, we consider equations which are non-local in both time and space:
\begin{align*}
    \partial_t^\alpha u - Lu + \lambda u = f \qx{in} \rr^d_T,
\end{align*}
where $L$ is defined in \eqref{op}.
We impose the following assumptions on the kernel $K = K(t,x,y) > 0$.
\begin{assumption}\label{aaa1}
\begin{enumerate}
    \item There exist some $0<\nu< \Lambda$ such that
    \begin{align*}
    (2-\sigma) \frac{\nu}{|y|^{d+\sigma}} \le K(t,x,y) \le(2-\sigma)  \frac{\Lambda}{|y|^{d+\sigma}} .\stepcounter{equation} \tag{\theequation}\label{ass1}
\end{align*}
\item When $\sigma = 1$,
\[
\int_{\partial{B_r}} y K(t,x,y) \, dS_r(y) = 0. \stepcounter{equation} \tag{\theequation}\label{sig1}
\]
\end{enumerate}
\end{assumption}
Note that \eqref{sig1} implies that $\chi^1$ in \eqref{op} can be replaced by $1_{y\in B_r}$ for any $r>0$.

\begin{assumption}\label{aaa2}
There exist $\beta \in (0,1)$ and a continuous increasing function $\omega:\rr^+ \to \rr^+$ such that
\[
|y|^{d+\sigma}|K(t,x_1,y) - K(t,x_1,y)|  \le \omega(|x_1-x_2|)
\]
and
\[
\int_{|y|\le 1 }\omega(|y|)|y|^{-d-\beta} \, dy < \infty \stepcounter{equation} \tag{\theequation}\label{iou}.
\]
\end{assumption}

\begin{remark}
{On one hand, i}t is easily seen that if $|y|^{d+\sigma}K(t,\cdot,y)$ is H\"older continuous in $x$ with exponent $\beta + \ep$ for any $\ep>0$  {uniformly in $(t,y)$}, then \eqref{iou} is satisfied. {On the other hand, since $\omega$ is increasing, from \eqref{iou} we see that $\omega(r)\le Nr^\beta$ for any $r\le 1/2$, which implies that $|y|^{d+\sigma}K(t,\cdot,y)$ is H\"older continuous in $x$ with exponent $\beta$ uniformly in $(t,y)$.}
\end{remark}

\begin{remark}\label{re}
{It is well known that} for $\sigma\in (0,2)$, if $K(t,x,y) = K(y) := c|y|^{-d-\sigma}$, where $$c = c(d,\sigma) = \frac{\sigma(2-\sigma)\Gamma(\frac{d+\sigma}{2})}{\pi^{d/2}2^{2-\sigma}\Gamma(2-\frac{\sigma}{2})},$$
then
\[
{Lu}= \frac{c}{2} \int_{\rr^d} \bigg( u(t,x+y)-u(t,x) - 2u(x)\bigg)|y|^{-d-\sigma}\,dy{=-(-\Delta)^{\sigma/2}u}.
\]
\end{remark}

In some cases, the lower bound of the operator is not necessary. Thus, we use $L^1$ to denote operators satisfying \eqref{sig1} {when $\sigma=1$} with the kernel
\begin{align*}
     K_1(t,x,y) = K_1(t,y) \le(2-\sigma)  \frac{\Lambda}{|y|^{d+\sigma}} \stepcounter{equation} \tag{\theequation}\label{ass}.
\end{align*}

With the assumptions {above}, we are ready to state the main theorems of this paper.
\begin{theorem}\label{main}
Let $\alpha \in (0,1]$, $\sigma\in(0,2)$, $T \in (0, \infty)$, $\lambda \ge 0$, and $p \in (1, \infty)$. Suppose {that} the kernel $K= K(t,y)$ satisfies Assumption \ref{aaa1}. Then $\partial^\alpha_t - L  $ is a continuous operator from $\hh_{p,0}^{\alpha, \sigma} (T )$ to $L_p(\rr^d_T)$. Also, for any $u \in \hh_{p,0}^{\alpha, \sigma} (T )$ satisfying
\[
\partial_t^\alpha u - Lu + \lambda u = f \qx{in} \rr^d_T, \stepcounter{equation} \tag{\theequation}\label{eqn}
\]
and any $L^1$ satisfying \eqref{sig1} {when $\sigma=1$} and \eqref{ass}, we have
\[
\norm{\partial_t^\alpha u}_{L_p(\rr^d_T)} + \norm{L^1 u}_{L_p(\rr^d_T)} + \lambda \norm{ u}_{L_p(\rr^d_T)}
\leq N \norm{f}_{L_p(\rr^d_T)} \stepcounter{equation}\tag{\theequation}\label{ess}
\]
and
\[
\norm{u}_{\hh_{p}^{\alpha, \sigma}(T)}
\leq N \min (T^\alpha, \lambda^{-1}) \norm{f}_{L_p(\rr^d_T)},\stepcounter{equation}\tag{\theequation}\label{es}
\]
where $N = N(d, \nu, \Lambda, \alpha,\sigma, p)$.
Moreover, for any $f \in L_p(\rr^d_T)$, there exists a unique solution $u \in \hh_{p,0}^{\alpha, \sigma} ( T )$ to \eqref{eqn}.
\end{theorem}
{\begin{remark}\label{2.3}
Indeed, \eqref{es} follows from \eqref{ess} upon setting $L^1=(-\Delta)^{\sigma/2}$, $L^1=L$, and using \eqref{eqn} and Lemma \ref{a2}. Furthermore,
by \eqref{ess} with $\lambda = 0$, for any $u \in \hh_{p,0}^{\alpha, \sigma} (T )$ if
\[
\partial_t^\alpha u + \pf u  = g \qx{in} \rr^d_T,
\]
then
\[
\norm{L^1 u}_{L_p(\rr^d_T)}
\leq N \norm{g}_{L_p(\rr^d_T)}. \stepcounter{equation}\tag{\theequation}\label{r}
\]
For any $v \in H_{p}^\sigma(\rr^d)$, by first taking $T>0$ and {a nonzero function} $\eta \in C_0^\infty(\rr)$ with $\eta(0) = 0$, then applying \eqref{r} to $u(t,x):= \eta(t/T)v(x)$, and finally sending $T\to \infty$,
we conclude that
\[
\|L^1v \|_{L_p(\rr^d)} \leq N \| \pf v \|_{L_p(\rr^d)}.\stepcounter{equation}\tag{\theequation}\label{hj}
\]
Thus, for $u \in \hh_{p,0}^{\alpha, \sigma} (T )$ and $t\in(0,T)$
\[
\|L^1u(t,\cdot) \|_{L_p(\rr^d)} \leq N \| \pf u (t,\cdot) \|_{L_p(\rr^d)},
\]
which implies
\[
\|L^1u \|_{L_p(\rr^d_T)} \leq N \| \pf u \|_{L_p(\rr^d_T)} \stepcounter{equation}\tag{\theequation}\label{up}
\]
and the continuity of the operator $\partial^\alpha_t - L $.
We refer the reader to \cite{dong12} for a different proof of \eqref{hj}.
\end{remark}}

Next, we have the following results when $ w = w_1(t)$, $K= K(t,x,y)$,
and
\[
|b |=\big|\big(b^1(t,x),\ldots,b^d(t,x)\big)\big|\le M, \quad  |c| = |c(t,x)| \le M. \stepcounter{equation} \tag{\theequation}\label{up77}
\]

\begin{corollary}\label{newco}
Let $\beta\in(0,1)$, $\alpha \in (0,1]$, $\sigma\in(0,2)$, $T \in (0, \infty)$, and $p \in (1, \infty)$. Suppose that the kernel $K= K(t,x,y)$ satisfies Assumptions \ref{aaa1} and \ref{aaa2}. Then $\partial^\alpha_t - L  $ is a continuous operator from $\hh_{p,0}^{\alpha, \sigma} (T )$ to $L_p(\rr^d_T)$.
\begin{enumerate}
    \item There exists $\lambda_0 = \lambda_0(d, \nu, \Lambda, \alpha,\sigma, p,M,\beta,\omega)\ge 1$ such that for any $\lambda \ge \lambda_0$ and $u\in \hh_{p,0}^{\alpha, \sigma} (T )$ satisfying
\[
\partial_t^\alpha u - Lu +  \lambda u= f \qx{in} \rr^d_T,
\]
we have
\[
\norm{u}_{\hh_{p}^{\alpha, \sigma}(T)}
\leq N \norm{f}_{L_p(\rr^d_T)},
\]
where $N = N(d, \nu, \Lambda, \alpha,\sigma, p,M,\beta,\omega)$ is independent of $T$.
\item
Also, for any $u \in \hh_{p,0}^{\alpha, \sigma} (T )$ satisfying
\[
\partial_t^\alpha u - Lu  + b^i D_iu 1_{\sigma > 1} + cu= f \qx{in} \rr^d_T,  \stepcounter{equation}\tag{\theequation}\label{up5}
\]
we have
\[
\norm{u}_{\hh_{p}^{\alpha, \sigma}(T)}
\leq N \norm{f}_{L_p(\rr^d_T)},
\stepcounter{equation}\tag{\theequation}\label{up6}
\]
where $N = N(d, \nu, \Lambda, \alpha,\sigma, p,T,M,\beta,\omega)$.
Moreover, for any $f \in L_p(\rr^d_T)$, there exists a unique solution $u \in \hh_{p,0}^{\alpha, \sigma} ( T )$ to \eqref{up5}.
Furthermore, when $\sigma = 1$, if $\norm{b}_{\infty}$ is sufficiently small, then the a priori estimate and the unique solvability hold for
\[
\partial_t^{\alpha} u - Lu  + b^i D_iu  + cu= f \qx{in} \rr^d_T.\stepcounter{equation}\tag{\theequation}\label{up66}
\]
When $\alpha=1$ and $\sigma = 1$, if $b$ is uniformly continuous, then the a priori estimate and the unique solvability also hold for \eqref{up66}.
\end{enumerate}
\end{corollary}

In the case of $\alpha = 1$, i.e., the operator is local in time, we have the following results regarding the weighted mixed-norm.
\begin{theorem}\label{main2}
Let $\sigma\in(0,2)$, $T \in (0, \infty)$, $\lambda \ge 0$, $p,q \in (1, \infty)$, $M_1 \in [1, \infty)$, and $[w]_{p,q} \le M_1$. Suppose the kernel $K= K(t,y)$ satisfies Assumption \ref{aaa1}. Then $\partial_t - L $ is a continuous operator from $\hh_{p,q,w,0}^{1,\sigma} ( T )$ to $L_{p,q,w}(\rr^d_T)$. Also, for any $u \in \hh_{p,q,w,0}^{1,\sigma} ( T )$ satisfying
\[
\partial_t u - Lu + \lambda u = f \qx{in} \rr^d_T
\stepcounter{equation} \tag{\theequation}\label{eqn6}
\]
and any $L^1$ satisfying \eqref{sig1} {when $\sigma=1$} and \eqref{ass}, we have
\[
\norm{\partial_t u}_{L_{p,q,w}(\rr^d_T)}+ \norm{L^1 u}_{L_{p,q,w}(\rr^d_T)}+ \lambda \norm{ u}_{L_{p,q,w}(\rr^d_T)}
\leq N \norm{f}_{L_{p,q,w}(\rr^d_T)}
\stepcounter{equation}\tag{\theequation}\label{lo1}
\]
and
\[
\norm{u}_{\hh_{p,q,w}^{1,\sigma} ( T )}
\leq N \min (T, \lambda^{-1}) \norm{f}_{L_{p,q,w}(\rr^d_T)},
\stepcounter{equation}\tag{\theequation}\label{uo1}
\]
where $N = N(d, \nu, \Lambda, \alpha,\sigma, p, q,M_1)$.
Moreover, for any $f \in L_{p,q,w} ( T )$, there exists a unique solution $u \in \hh_{p,q,w,0}^{1,\sigma} ( T )$ to \eqref{eqn6}.
\end{theorem}

With bounded $b$ and $c$ satisfying \eqref{up77}, we have the following result.
\begin{corollary}\label{newco1}
Let $\beta\in(0,1)$, $\sigma\in(0,2)$, $T \in (0, \infty)$,  $p,q \in (1, \infty)$,  $M_1 \in [1, \infty)$, and $[w]_{p,q} \le M_1$. Suppose {that} the kernel $K= K(t,x,y)$ satisfies Assumptions \ref{aaa1} and \ref{aaa2}. Then $\partial^\alpha_t - L $ is a continuous operator from $\hh_{p,q,w,0}^{1, \sigma} (T )$ to $L_{p,q,w}(\rr^d_T)$.
\begin{enumerate}
    \item There exists $\lambda_0 = \lambda_0(d, \nu, \Lambda, \sigma, p,M_1,M,\beta,\omega)\ge 1$ such that for any $\lambda \ge \lambda_0$ and $u\in \hh_{p,0}^{\alpha, \sigma} (T )$ satisfying
\[
\partial_t u - Lu +  \lambda u= f \qx{in} \rr^d_T,
\]
we have
\[
\norm{u}_{\hh_{p,q,w}^{1, \sigma}(T)}
\leq N \norm{f}_{L_{p,q,w}(\rr^d_T)}, \stepcounter{equation}\tag{\theequation}\label{89oii}
\]
where $N =N(d, \nu, \Lambda, \sigma, p,M_1,M,\beta,\omega)$ is independent of $T$.
\item Also, for any $u \in \hh_{p,q,w,0}^{1, \sigma} (T )$ satisfying
\[
\partial_t u - Lu  + b^i D_iu  1_{\sigma >1} + cu= f \qx{in} \rr^d_T,  \stepcounter{equation}\tag{\theequation}\label{upp5}
\]
we have
\[
\norm{u}_{\hh_{p,q,w}^{1, \sigma}(T)}
\leq N \norm{f}_{L_{p,q,w}(\rr^d_T)},\stepcounter{equation}\tag{\theequation}\label{89oi}
\]
where $N = N(d, \nu, \Lambda, \sigma, p,T,M_1,M,\beta,\omega)$.
Moreover, for any $f \in L_{p,q,w} ( T )$, there exists a unique solution $u \in \hh_{p,q,w,0}^{1, \sigma} ( T )$ to \eqref{upp5}. {Furthermore, when $\sigma = 1$, if $\norm{b}_{\infty}$ is sufficiently small or $b$ is uniformly continuous, then the a priori estimate and the unique solvability hold for
\[
\partial_t u - Lu  + b^i D_iu  + cu= f \qx{in} \rr^d_T.
\]}
\end{enumerate}
\end{corollary}

\begin{remark}
                    \label{rem2.6}
For $\sigma\in(0,2)$, the constant $N$ in above Theorems and Corollaries can be chosen to be dependent on $\Tilde{\sigma } $ instead of on $\sigma$ itself,
where
\[\Tilde{\sigma } = \begin{cases}(\sigma_0,\sigma_1) \quad  &\text{when}\,\,   0 < \sigma_0 \le \sigma \le \sigma_1 < 1, \\
1  \quad &\text{when}\,\,   \sigma =1,\\
\sigma_0 \quad &\text{when}\,\,  1 < \sigma_0 \le \sigma < 2.\end{cases}
\stepcounter{equation}\tag{\theequation}\label{bdds}\]
Thus, when $\sigma > 1$, $N$ does not blow up when $\sigma \nearrow 2$. {To see this, w}e keep track of the dependence of constants on $\sigma$ in Lemmas \ref{com} and the proof of Theorem \ref{main} when $p =2$. Moreover, note that Lemmas \ref{a6} and \ref{aa7} hold for $\sigma_0$, and
$$\norm{u}_{\hh_{p}^{\alpha, \sigma_0}(T)} \le N\norm{u}_{\hh_{p}^{\alpha, \sigma}(T)},$$
where $N$ can be chosen to be independent of $\sigma$ by the Mikhlin multiplier theorem {(see, for instance, \cite{MR542885}) and \cite[Lemma 3.4]{dyk}}.
Thus, we can choose the increment of $p$ in the iteration arguments in the proofs of Theorems \ref{main} and \ref{main2} using $\sigma_0$ instead of $\sigma$. {Similar phenomena were observed before, for example in \cite{MR2494809} and \cite{dong12}.}
\end{remark}

\section{Equations in \texorpdfstring{$L_p$}p}\label{3}
In this section, we prove Theorem \ref{main} and Corollary \ref{newco}. To prove Theorem \ref{main}, we use a level set argument and a bootstrap argument with the Sobolev embedding.

\subsection{The case of \texorpdfstring{$p=2$}p and auxiliary results}
In order to apply the bootstrap argument, we start with the case when $p = 2$.
\begin{proof}[Proof of Theorem \ref{main} when $p =2$]
We first prove the continuity of $\partial^\alpha_t - L  $.
By taking the Fourier transform,
\begin{align*}
 \widehat{Lu}(t,\xi) = \hat{u}(t, \xi)
\int_{\rr^d} (e^{i \xi \cdot y} - 1 - i y \cdot \xi \chi^{(\sigma)} (y))
K(t,y) \, dy
:= \hat{u}(t, \xi) m(\xi)
,\stepcounter{equation}\tag{\theequation}\label{fo4}
\end{align*}
and in particular, by Remark \ref{re},
\begin{align*}
 \widehat{\pf u}(t,\xi) =\hat{u}(t, \xi)
\int_{\rr^d} \big(1 - \cos (\xi \cdot y)\big)  c|y|^{-d-\sigma} \, dy.  \stepcounter{equation} \tag{\theequation}\label{fo1}
\end{align*}
By a change of variables $y\to y/|\xi|$ and the upper bound of $K$, it is seen that
\[
|m(\xi)| \le N(d, \Lambda, \tilde\sigma)|\xi|^\sigma,
%
\]
{where $\tilde\sigma$ is defined in Remark \ref{rem2.6}.}
Thus, \eqref{fo4} leads to
\[ \norm{\widehat{Lu}}_{L_2(\rr^d_T)} \le N \norm{\widehat{u}(\cdot,\xi)|\xi|^{\sigma}}_{L_2(\rr^d_T)}= N\norm{\pf u}_{L_2(\rr^d_T)},\stepcounter{equation}\tag{\theequation}\label{oio}\]
where $N = N(d, \Lambda, \tilde\sigma)$.

Next, we prove the a priori estimate \eqref{ess}. With the continuity of the operator and the density of smooth functions in $\hh_{2,0}^{\alpha, \sigma} (T )$, without loss of generality, we assume that $u \in C_0^\infty ([0,T]\times \rr^d)$ with $u(0, \cdot) = 0$.
Multiplying $(-\Delta)^{\sigma/2} u$ to both sides of \eqref{eqn} and integrating over $\rr^d_T$, we arrive at
\begin{align*}
  \int_{\rr^d_T} \partial_t^\alpha u  \pf u - \int_{\rr^d_T} Lu\pf u + \lambda\int_{\rr^d_T} u\pf u  = \int_{\rr^d_T} f\pf u
\stepcounter{equation}\tag{\theequation}\label{l2}.
\end{align*}

For the first term on the left-hand side of \eqref{l2}, due to \eqref{fo1} and the fact that it is real,
\begin{align*}
&\int_{\rr^d_T} \partial_t^\alpha u(t, x) \pf u(t, x) \, dx \, d t = \int_{\rr^d_T} \widehat{\pf u}(t, \xi) \overline{\partial_t^\alpha\widehat{ u}(t, \xi)} \, d\xi \, dt \\
&=c \int_{\rr^d_T} \overline{\partial_t^\alpha\widehat{ u}(t, \xi)}
\hat{u}(t, \xi)
\int_{\rr^d} \big(1 - \cos (\xi \cdot y)\big)  |y|^{-d-\sigma} \, dy \, d \xi  \, dt\\
&=c \int_{\rr^d_T} \partial_t^\alpha \real (\widehat{ u})(t, \xi)
\real(\hat{u})(t, \xi)
\int_{\rr^d} \big(1 - \cos (\xi \cdot y)\big)  |y|^{-d-\sigma} \, dy \, d \xi  \, dt\\
&\quad + c \int_{\rr^d_T} \partial_t^\alpha \imag (\widehat{ u})(t, \xi)
\imag(\hat{u})(t, \xi)
\int_{\rr^d} \big(1 - \cos (\xi \cdot y)\big)  |y|^{-d-\sigma} \, dy \, d \xi  \, dt.
\stepcounter{equation}\tag{\theequation}\label{fo3}
\end{align*}
Also, by \cite[Proposition 4.1]{dong19}, we have
\begin{align*}
   & \partial_t^\alpha \real (\widehat{ u})(t, \xi) \real(\hat{u})(t, \xi) + \partial_t^\alpha \imag (\widehat{ u})(t, \xi)
\imag(\hat{u})(t, \xi) \\
&\ge \frac{1}{2}\partial_t^\alpha | \real (\widehat{ u})|^2 (t, \xi)+ \frac{1}{2}\partial_t^\alpha | \imag (\widehat{ u})|^2(t, \xi) =  \frac{1}{2}\partial_t^\alpha | \hat{u}|^2(t, \xi) \stepcounter{equation}\tag{\theequation}\label{foo3}.
\end{align*}
Thus, by \eqref{fo3}, \eqref{foo3}, and the definition of $\partial_t^\alpha$,
\begin{align*}
&\int_{\rr^d_T} \partial_t^\alpha u(t, x) \pf u(t, x) \, dx \, d t \\
&\geq c\int_{\rr^d_T} \frac{1}{2} \partial_t^\alpha | \hat{u}|^2(t, \xi)
\int_{\rr^d} \big(1 - \cos (\xi \cdot y)\big)  |y|^{-d-\sigma} \, dy  \, d \xi \, dt \\
&= N(d, \sigma) \int_{\rr^d} \int_{\rr^d} \big( 1-\cos (\xi \cdot y)\big) |y|^{-d- \sigma} \int_0^T \partial_t^\alpha |\hat{u}|^2(t, \xi)  \, dt \, dy \, d \xi \\
&=
N(d, \alpha,\sigma )
\int_{\rr^d} \int_{\rr^d} \big(1-\cos (\xi \cdot y)\big) |y|^{-d - \sigma}
\int_0^T (T-s)^{-\alpha}
| \hat{u} (s, \xi) |^2 \, ds \, dy  \, d \xi \ge 0.
\end{align*}
For the second term on the left-hand side of \eqref{l2}, by \eqref{fo1} and \eqref{fo4},
\begin{align*}
&-\int_{\rr^d_T} Lu(t,x) \pf u(t, x)  \, dx \, dt=- \int_{\rr^d_T}  \overline{\widehat{Lu}(t,\xi)} \widehat{\pf u}(t, \xi) \, d\xi \, dt \\
&= \int_{\rr^d_T}  | \hat{u}(t, \xi) |^2
\bigr( \int_{\rr^d} \big(1 - \cos (\xi \cdot y)\big)  K (t,y) \, dy \bigr) \int_{\rr^d} \big(1 - \cos (\xi \cdot y)\big)  c|y|^{-d-\sigma} \, dy  \, d \xi \, dt \\
&\ge c^{-1} (2-\sigma) \nu \int_{\rr^d_T}
| \widehat{\pf u}(t, \xi)|^2  \, d \xi \, dt \\
&= N(d, \nu) \| (-\Delta)^{\sigma/2} u\|_{L_2(\rr^d_T)}^2, \stepcounter{equation}\tag{\theequation}\label{fo2}
\end{align*}
where {in the second equality we used the fact that the left-hand side is real, and here} $N$ is independent of $\sigma$ since for $\sigma\in(0,2)$, $c^{-1} (2-\sigma)  \ge N(d)$ by the continuity of the $\Gamma$ function and the definition of $c$ in Remark \ref{re}.
The third term on the left-hand side of \eqref{l2} is nonnegative.
{Similarly, multiplying $\lambda u$ to both sides of \eqref{eqn} and integrating over $\rr^d_T$, we get
\begin{equation}
                        \label{eq5.01}
\lambda^2 \int_{\rr^d_T} u^2\le \int_{\rr^d_T} f^2.
\end{equation}}

Thus, \eqref{l2}, \eqref{fo3}, \eqref{fo2}, {and \eqref{eq5.01}} together with Young's inequality yield
\[
\| (-\Delta)^{\sigma/2} u\|_{L_2(\rr^d_T)} + \lambda \| u\|_{L_2(\rr^d_T)}
\leq N(d, \nu) \norm{f}_{L_2(\rr^d_T)}.
\]
Together with \eqref{oio}, we arrive at
\begin{align*}
\||L^1 u| + \lambda |u|\|_{L_2(\rr^d_T)}
&\le N(d,\Lambda,\tilde\sigma)\||\pf u| + \lambda |u|\|_{L_2(\rr^d_T)} \\
&\leq N(d, \nu,\Lambda,\tilde\sigma) \norm{f}_{L_2(\rr^d_T)},
\end{align*}
where $L^1$ is defined in \eqref{ass}.
Finally, by \cite[Theorem 2.8]{kim20} together with the method of continuity, we have the existence of solutions. Theorem \ref{main} is proved when $p=2$.
\end{proof}

Next, assuming that Theorem \ref{main} holds,  we derive a local estimate. In the following lemma, we denote $\norm{\, \cdot\, }_{p,r} := \norm{\,\cdot\,}_{L_p((0,T) \times B_r)}$ for any $r> 0 $. Also, recall the definition of $ \norm{\cdot}_{L_p((0,T);L_1(\rr^d,\psi))}$ in \eqref{itr}.
\begin{lemma}\label{local}
Let $\alpha \in (0,1]$, $\sigma\in(0,2)$, $T \in (0, \infty)$, $p  \in (1, \infty)$, and $0<r<R< \infty$. Also, let $\zeta_0 $ be a cutoff function such that
$$\zeta_0 \in C_0^{\infty}(B_{(R+r)/2}), \quad \zeta_0 = 1 \qx{in} B_{r},\qx{and}|D\zeta_0|\le  4/(R-r).$$
Suppose that Theorem \ref{main} holds {for this $p$}. If $u \in {\hh_{p,0}^{\alpha,\sigma}(T)}$, and $L^1$ satisfies \eqref{sig1} {when $\sigma=1$} and \eqref{ass}, and
\begin{align*}
    \partial_t^\alpha u- Lu + \lambda u= f \qx{in} \rr^d_T, \stepcounter{equation}\tag{\theequation}\label{es22}
\end{align*}
then
\begin{align*}
& \norm{\partial_t^\alpha  (\zeta_0 u)}_{L_p(\rr^d_T)}  + \norm{L^1(\zeta_0 u)}_{L_p(\rr^d_T)} + \lambda \norm{\zeta_0 u}_{L_p(\rr^d_T)}
\\&\le N \norm{ f}_{p,R}+  N  \frac{\norm{u}_{p,R}}{(R-r)^{\sigma}}   + N \frac{R^{d/p}(1 + R^{d+\sigma} )}{(R-r)^{d+\sigma}}\norm{u}_{L_p((0,T);L_1(\rr^d,\psi))},\stepcounter{equation}\tag{\theequation}\label{fe}
\end{align*}
and
\begin{align*}
&\norm{\partial_t^\alpha u }_{p,r} + \norm{L^1 u}_{p,r} +  \lambda \norm{ u}_{p,r}  \\
&\le N \norm{ f}_{p,R}+  N  \frac{\norm{u}_{p,R}}{(R-r)^{\sigma}}   + N \frac{R^{d/p}(1 + R^{d+\sigma} )}{(R-r)^{d+\sigma}}\norm{u}_{L_p((0,T);L_1(\rr^d,\psi))}, \stepcounter{equation}\tag{\theequation}\label{es2}
\end{align*}
where $N = N(d,\nu,\Lambda,\alpha,\sigma, p)$.
\end{lemma}
\begin{proof}
We first take cutoff functions in the spatial variables as follows. For $k = 1,2, \ldots$, let
$$ r_k = r + (R-r)\sum_{j = 1}^{k}2^{-j},\quad \zeta_k \in C_0^{\infty}(B_{r_{k+1}}),$$
\[ \zeta_k\in[0,1], \quad \zeta_k = 1 \qx{in} B_{r_k}, \quad |D\zeta_k|\le  \frac{4\cdot2^{k}}{R-r}, \qx{and} |D^2\zeta_k|\le \frac{16\cdot2^{2k}}{(R-r)^2}. \stepcounter{equation}\tag{\theequation}\label{cutt} \]
It follows that $\zeta_k u \in {\hh_{p,0}^{\alpha,\sigma}(T)}$ and
\begin{align*}
\partial_t^\alpha (\zeta_k u)- L(\zeta_k u) + \lambda (\zeta_k u) = \zeta_kf +\zeta_k Lu -L(\zeta_k u)   \qx{in} \rr^d_T.
\stepcounter{equation}\tag{\theequation}\label{cut}
\end{align*}
By \eqref{up},
\begin{equation}
                \label{6}
\begin{aligned}
&{\norm{\partial_t^\alpha u}_{p,r} }+ \norm{L^1u}_{p,r_k} + \lambda \norm{ u}_{p,r}\\
&\le \norm{\partial_t^\alpha (\zeta_k u)}_{p,r} + \norm{\zeta_kL^1u}_{p,r_k}+ \lambda \norm{ \zeta_k u}_{p,r}\\
&\le   \norm{L^1(\zeta_k u)-\zeta_kL^1u  }_{L_p(\rr^d_T)} + \norm{|\partial_t^\alpha  (\zeta_k u)|+|L^1(\zeta_k u)| + \lambda|\zeta_k u|}_{L_p(\rr^d_T)} \\
&\le   \norm{L^1(\zeta_k u)-\zeta_kL^1u  }_{L_p(\rr^d_T)} + N \norm{|\partial_t^\alpha  (\zeta_k u)|+|\pf(\zeta_k u)|+ \lambda|\zeta_k u|}_{L_p(\rr^d_T)} .
\end{aligned}
\end{equation}
Moreover, by applying Theorem \ref{main} to \eqref{cut}, we have
\begin{align*}
&\norm{|\partial_t^\alpha  (\zeta_k u)|+|\pf(\zeta_k u)|+ \lambda|\zeta_k u|}_{L_p(\rr^d_T)} \\
&\le N \norm{\zeta_kf +\zeta_k Lu -L(\zeta_k u) }_{L_p(\rr^d_T)}\\
&\le N \norm{ f}_{p,R}  + N \norm{\zeta_k Lu -L(\zeta_k u) }_{L_p(\rr^d_T)} , \stepcounter{equation}\tag{\theequation}\label{6po}
\end{align*}
where {$N = N(d, \nu, \Lambda, \alpha,\sigma, p)$}.
We refer the reader to Lemma \ref{com} for the details about the estimates of {the commutator term} $\norm{\zeta_k Lu -L(\zeta_k u) }_{L_p(\rr^d_T)}$. Note that since we only used the upper bound of the kernel of the operator in Lemma \ref{com}, the estimate can be applied to $ \norm{ L^1(\zeta_k u) - \zeta_kL^1u  }_{L_p(\rr^d_T)}$

{\em Case 1: $\sigma \in (0,1)$.} In this case, \eqref{es2} and \eqref{fe} follow directly from \eqref{6}, \eqref{6po} and \eqref{a41} with $k = 0$.

{\em Case 2: $\sigma \in (1,2) $.} In this case, by \eqref{6po} and \eqref{a42}, we have
\begin{align*}
&\norm{\partial_t^\alpha  (\zeta_k u)}_{L_p(\rr^d_T)}+ \norm{\pf(\zeta_k u)}_{L_p(\rr^d_T)} + \lambda \norm{\zeta_k u}_{L_p(\rr^d_T)} \\
&\le N \norm{f}_{p,R} + N \frac{2^{(\sigma-1) k}}{(R-r)^{\sigma-1}} \norm{Du}_{p,r_{k+3}} + N  \frac{2^{\sigma k}}{(R-r)^{\sigma}} \norm{u}_{p,R}  \\
&\quad+ N \frac{2^{(d+\sigma)k}}{(R-r)^{d+\sigma}}R^{d/p}(1 + R^{d+\sigma} ) \norm{u}_{L_p((0,T);L_1(\rr^d,\psi))}. \stepcounter{equation}\tag{\theequation}\label{p03}
\end{align*}
Also, note that by Lemma \ref{a1}, for any $\ep \in(0,1)$,
\begin{align*}
& N \frac{2^{(\sigma-1) k}}{(R-r)^{\sigma-1}} \norm{Du}_{p,r_{k+3}} \le  N \frac{2^{(\sigma-1) k}}{(R-r)^{\sigma-1}} \norm{D(\zeta_{k+3} u)}_{L_p(\rr^d_T)} \\
&\le N \frac{2^{\sigma k}}{(R-r)^{\sigma}} \ep^{3/(1- \sigma)}\norm{\zeta_{k+3} u}_{L_p(\rr^d_T)} + \ep^3 \norm{(-\Delta)^{\sigma/2}(\zeta_{k+3} u)}_{L_p(\rr^d_T)}\\
&\le  N  \frac{2^{\sigma k}}{(R-r)^{\sigma}}  \ep^{3/(1- \sigma)}\norm{u}_{p,R} + \ep^3 \norm{(-\Delta)^{\sigma/2}(\zeta_{k+3} u)}_{L_p(\rr^d_T)}.
\stepcounter{equation}\tag{\theequation}\label{ee4}
\end{align*}
Therefore, \eqref{p03} and \eqref{ee4} lead to
\begin{align*}
&\norm{\partial_t^\alpha  (\zeta_k u)}_{L_p(\rr^d_T)}+ \norm{(-\Delta)^{\sigma/2}(\zeta_k u)}_{L_p(\rr^d_T)} + \lambda \norm{\zeta_k u}_{L_p(\rr^d_T)}\\\
&\le N \norm{f}_{p,R}  +  N  \frac{2^{\sigma k}}{(R-r)^{\sigma}}  \ep^{3/(1- \sigma)}\norm{u}_{p,R} + \ep^3 \norm{(-\Delta)^{\sigma/2}(\zeta_{k+3} u)}_{L_p(\rr^d_T)} \\
&\quad+ N \frac{2^{(d+\sigma)k}}{(R-r)^{d+\sigma}}R^{d/p}(1 + R^{d+\sigma} ) \norm{u}_{L_p((0,T);L_1(\rr^d,\psi))}. \stepcounter{equation}\tag{\theequation}\label{ee5}
\end{align*}
Multiplying $\ep^k$ to both sides of \eqref{ee5} and taking the sum over $k = 0,1, \ldots$ lead to
\begin{align*}
&\norm{|\partial_t^\alpha  (\zeta_0 u)| + |\lambda  (\zeta_0 u)|}_{L_p(\rr^d_T)}\sum_{k=0}^\infty \ep^k  + \sum_{k=0}^\infty \ep^k \norm{(-\Delta)^{\sigma/2}(\zeta_k u)}_{L_p(\rr^d_T)} \\
&\le  N  \frac{ \ep^{3/(1- \sigma)}}{(R-r)^{\sigma}} \norm{u}_{p,R} \sum_{k=0}^\infty (\ep 2 ^\sigma)^k  + \sum_{k=0}^\infty \ep^{k+3}\norm{(-\Delta)^{\sigma/2}(\zeta_{k+3} u)}_{L_p(\rr^d_T)} \\
&\quad+N \norm{f}_{p,R} \sum_{k=0}^\infty \ep^k   + N \frac{R^{d/p}(1 + R^{d+\sigma} )}{(R-r)^{d+\sigma}} \norm{u}_{L_p((0,T);L_1(\rr^d,\psi))}\sum_{k=0}^\infty (\ep2^{d+\sigma})^k . \stepcounter{equation}\tag{\theequation}\label{ee6}
\end{align*}
Therefore, by first taking $\ep$ to be sufficiently small so that $\ep2^{d+\sigma}< 1$, and then absorbing the second term on the right-hand side of \eqref{ee6} to the left-hand side, we arrive at
\eqref{fe} with $\pf$ in place of $L^1$ on the left-hand side. The estimates for general $L^1$ follow from \eqref{up}.
Similarly, it is easily seen that if we replace $\zeta_0 u$ with $\zeta_3 u$ on the left-hand side of \eqref{fe}, the inequality still holds. Thus, by  \eqref{a42}, \eqref{up}, and \eqref{ee4} with $k= 0$,
\begin{align*}
&\norm{L^1u}_{p,r} \le  \norm{L^1(\zeta_0 u)}_{L_p(\rr^d_T)}
+\norm{L^1(\zeta_0 u)-\zeta_0L^1u  }_{L_p(\rr^d_T)} \\
& \le \norm{(-\Delta)^{\sigma/2}(\zeta_0 u)}_{L_p(\rr^d_T)} +  N  \frac{\norm{u}_{{p,}R}}{(R-r)^{\sigma}}   +  \norm{(-\Delta)^{\sigma/2}(\zeta_{3} u)}_{L_p(\rr^d_T)}   \\
&\quad + N \frac{R^{d/p}(1 + R^{d+\sigma} )}{(R-r)^{d+\sigma}} \norm{u}_{L_p((0,T);L_1(\rr^d,\psi))} \\
&\le  N \norm{ f}_{p,R} + N \frac{\norm{ u}_{p,R}}{(R-r)^\sigma}+  N \frac{R^{d/p}(1 + R^{d+\sigma} )}{(R-r)^{d+\sigma}} \norm{u}_{L_p((0,T);L_1(\rr^d,\psi))}.
\end{align*}
Combining this with \eqref{fe}, we infer \eqref{es2} when $\sigma\in(1,2)$.

{\em Case 3: $\sigma = 1$.} In this case, for any $\ep\in(0,1)$ and $k\ge 1$, by \eqref{6}, \eqref{6po}, and \eqref{a43}, we have
\begin{align*}
\norm{|\partial_t^\alpha u |+ |\lambda u| }_{p,r_1}& + \norm{D u}_{p,r_k} \le N \norm{f}_{p,R} + \ep^3\norm{Du}_{p,r_{k+3}} +  N\frac{2^{k}}{(R-r)}\ep^{-3} \norm{u}_{p,R} \\
& +  N R^{d/p}\bigg(1+ \frac{2^{(d+1)k}\ep^{-3(d+1)}}{(R-r)^{d+1}}  (1 + R^{d+1} )\bigg) \norm{u}_{L_p((0,T);L_1(\rr^d,\psi))}. \stepcounter{equation}\tag{\theequation}\label{ee9}
\end{align*}
Multiplying $\ep^{k-1}$ to both sides of \eqref{ee9} and taking the sum over $k = 1, \ldots$ lead to
\begin{align*}
&\norm{|\partial_t^\alpha u| + |\lambda u| }_{p,r_1}\sum_{k=1}^\infty \ep^{k-1} + \sum_{k=1}^\infty \ep^{k-1} \norm{D u}_{p,r_k}\\
&\le N \norm{f}_{p,R} \sum_{k=1}^\infty \ep^{k-1}  +  \sum_{k=1}^\infty \ep^{k+2}\norm{Du}_{p,r_{k+3}}+  N\frac{\ep^{-4} \norm{u}_{p,R}}{(R-r)} \sum_{k=1}^\infty (\ep2)^k \\
&\,\, +  N R^{d/p} \norm{u}_{L_p((0,T);L_1(\rr^d,\psi))}\sum_{k=1}^\infty \ep^{k-1}\bigg(1+ \frac{2^{(d+1)k}\ep^{-3(d+1)}}{(R-r)^{d+1}}  (1 + R^{d+1} )\bigg). \stepcounter{equation}\tag{\theequation}\label{ee8}
\end{align*}
Therefore, by first taking $\ep$ to be sufficiently small so that $\ep2^{d+1}< 1$, and then absorbing the second term on the right-hand side of \eqref{ee8} to the left-hand side, we obtain
\begin{align*}
&\norm{\partial_t^\alpha u }_{p,r_1} + \norm{Du}_{p,r_1} + \lambda \norm{u}_{p,r_1} \\
&\le N \norm{ f}_{p,R}+  N  \frac{\norm{u}_{p,R}}{R-r}   + N \frac{R^{d/p}(1 + R^{d+1} )}{(R-r)^{d+1}}\norm{u}_{L_p((0,T);L_1(\rr^d,\psi))}.
\end{align*}
Also, by \eqref{up} and
\begin{align*}
&\norm{\partial_t^\alpha  (\zeta_0 u)}_{L_p(\rr^d_T)}  + \norm{D(\zeta_0 u)}_{L_p(\rr^d_T)} +  \lambda\norm{ \zeta_0 u}_{L_p(\rr^d_T)}\\
&\le \norm{\partial_t^\alpha  u}_{p,r_1}  + \norm{D u}_{p,r_1} {+N \frac{\norm{u}_{p,r_1}}{R-r} ,}
\end{align*}
we obtain \eqref{fe} when $\sigma=1$. As before, replacing $\zeta_0 {u}$ with $\zeta_3 {u}$ on the left-hand side of \eqref{fe} together with  \eqref{6} and \eqref{a43}, we obtain \eqref{es2}.
The lemma is proved.
\end{proof}

{\begin{remark}\label{qp}
In Lemma \ref{local}, we assume that $u \in {\hh_{p,0}^{\alpha,\sigma}(T)}$. However, in the proof we only used the fact that $\zeta_k u \in \hh_{p,0}^{\alpha,\sigma}(T)$. Thus,
it suffices to assume that $u|_{B_R} \in {\hh_{p,0}^{\alpha,\sigma}((0,T)\times B_R)}$ for $u$ defined on $(0,T)\times \rr^d$ satisfying \eqref{es22} {in $(0,T)\times B_R$}.
\end{remark}}

Lemma \ref{local} and the embeddings in Lemmas \ref{a5} and \ref{aa7} lead to the following corollary.
\begin{corollary}\label{so}
Let $\alpha \in (0,1]$, $\sigma\in(0,2)$, $T \in (0, \infty)$, $p  \in (1, \infty)$, $0<r<R< \infty$, $\lambda\ge 0$, $u \in L_p((0,T);L_1(\rr^d,\psi))$ be such that $u|_{{(0,T)\times} B_R} \in {\hh_{p,0}^{\alpha,\sigma}((0,T)\times B_R)}$, and $q\in (p,\infty)$ satisfy
\[
1/q= 1/p -\alpha\sigma/(\alpha d + \sigma).
\]
{Let $f=\partial_t^\alpha u- Lu+ \lambda u$ in $(0,T)\times B_R$.}
\begin{enumerate}
    \item  If $p \le d/\sigma + 1/\alpha$, then for any $l \in[p,q]$,
    \begin{align*}
     \norm{u}_{L_l((0,T)\times B_r)} & \le  N \norm{ f}_{L_p((0,T)\times B_R)}+  N  \frac{\norm{u}_{L_p((0,T)\times B_R)}}{(R-r)^{\sigma}}\\
     & \quad \quad + N \frac{R^{d/p}(1 + R^{d+\sigma} )}{(R-r)^{d+\sigma}}\norm{u}_{L_p((0,T);L_1(\rr^d,\psi))},
\end{align*}
where $N = N(d,\alpha,\sigma,p,l,T)$.
    \item If $p > d/\sigma + 1/\alpha$, then there exists {$\tau = \sigma -(d+\sigma/\alpha)/p \in (0,1)$} such that

     \begin{align*}
     \norm{u}_{C^{\tau\alpha/\sigma,\tau}((0,T)\times B_r)} & \le  N \norm{ f}_{L_p((0,T)\times B_R)}+  N  \frac{\norm{u}_{L_p((0,T)\times B_R)}}{(R-r)^{\sigma}}\\
     & \quad \quad + N \frac{R^{d/p}(1 + R^{d+\sigma} )}{(R-r)^{d+\sigma}}\norm{u}_{L_p((0,T);L_1(\rr^d,\psi))},
\end{align*}
where $N = N(d,\alpha,\sigma,p,T)$.
\end{enumerate}
\end{corollary}

\begin{proof}
(1) If $p \le d/\sigma + 1/\alpha$, recall the definition of $\zeta_0$ in Lemma \ref{local}. By the Sobolev embeddings in Lemma \ref{a5} and H\"older's inequality, we have
\begin{align*}
 &\norm{u}_{L_l((0,T)\times B_r)} \le   N \norm{\zeta_0 u}_{L_l(\rr^d_T)}  \le N \norm{\zeta_0 u}_{\hh_{p}^{\alpha,\sigma}(T)} \\
 &\le N \big( \norm{ f}_{L_p((0,T)\times B_R)}+   \frac{\norm{u}_{L_p((0,T)\times B_R)}}{(R-r)^{\sigma}}   + \frac{R^{d/p}(1 + R^{d+\sigma} )}{(R-r)^{d+\sigma}}\norm{u}_{L_p((0,T);L_1(\rr^d,\psi))}\big),
\end{align*}
where for the last inequality, we used \eqref{fe} and Remark \ref{qp}.

(2) If $p > d/\sigma + 1/\alpha$, the proof is similar to that of (1) by using Lemma \ref{aa7}.
\end{proof}

By Lemma \ref{local} with a scaling in the spatial coordinates, we derive an estimate that will be used later in Section \ref{4}.

\begin{corollary}\label{lr}
Let $\alpha \in (0,1]$, $\sigma\in(0,2)$, $T \in (0, \infty)$, $p  \in (1, \infty)$, and $R\in (0, \infty)$. Suppose that Theorem \ref{main} holds {for this $p$.} {If $u \in \hh_{p,0}^{\alpha,\sigma}(T)$}, and $L^1$ satisfies \eqref{sig1} {when $\sigma=1$} and \eqref{ass}, and
\begin{align*}
    \partial_t^\alpha u- Lu + \lambda u= f \qx{in} \rr^d_T,
\end{align*}
then
\begin{align*}
&\big( | \partial_t^\alpha u|^{p} \big)_{(0,T)\times B_{R/2}(x_0)}^{1/{p}} + \big( |L^1 u|^{p} \big)_{(0,T)\times B_{R/2}(x_0)}^{1/{p}} + \lambda \big( | u|^{p} \big)_{(0,T)\times B_{R/2}(x_0)}^{1/{p}}  \\
&\le N \big( | f|^{p} \big)_{(0,T)\times B_{R}(x_0)}^{1/{p}}+   N {R^{-\sigma}}\sum_{k = 0}^{\infty} 2^{-k\sigma} \big( |u|^p \big)_{(0,T)\times B_{2^kR}(x_0)}^{1/p}, \stepcounter{equation}\tag{\theequation}\label{ioio}
\end{align*}
where $N = N(d,\nu,\Lambda,\alpha,\sigma,p)$.
\end{corollary}
\begin{proof}
By shifting {the coordinates}, without loss of generality, we assume that $x_0 = 0$.

When $R=1$, denoting $p' = p/(p-1)$, {by the Minkowski inequality and H\"older's inequality,} we have
\begin{align*}
&\norm{u}_{L_p((0,T);L_1(\rr^d,\psi))}\\
&=\bigg( \int_{0}^{T} \big(\int_{\rr^d} |u(t,x)| \frac{1}{1+ |x|^{d+\sigma}}\,dx \big)^p\,dt \bigg)^{1/p}\\
&\le\int_{\rr^d} \big(\int_{0}^{T} |u(t,x)|^p \,dt\big)^{1/p} \frac{1}{1+ |x|^{d+\sigma}}\,dx  \\
&\le \sum_{k = 0}^{\infty} \int_{\widehat{B}_{2^k} \setminus \widehat{B}_{2^{k-1}}} \big(\int_{0}^{T} |u(t,x)|^p \,dt\big)^{1/p} \frac{1}{1+ |x|^{d+\sigma}}\,dx  \\
&\le \sum_{k = 0}^{\infty} \bigg(\int_{\widehat{B}_{2^k} \setminus \widehat{B}_{2^{k-1}}} \int_{0}^{T}
\frac{|u(t,x)|^p}{1+ |x|^{d+\sigma}} \,dt \,dx \bigg)^{1/p}  \bigg(\int_{\widehat{B}_{2^k} \setminus \widehat{B}_{2^{k-1}}} \frac{1}{1+ |x|^{d+\sigma}}\,dx \bigg)^{1/p'} \\
&\le T^{1/p} \sum_{k = 0}^{\infty} \bigg(2^{-k\sigma} \fint_{0}^{T} \fint_{B_{2^k}} | u(t,x)|^p \, dx \, dt\bigg)^{1/p} 2^{-k\sigma/p'} \\
&\le N T^{1/p} \sum_{k = 0}^{\infty} 2^{-k\sigma} \big( |u|^p \big)_{(0,T)\times B_{2^k}}^{1/p},\stepcounter{equation}\tag{\theequation}\label{opop}
\end{align*}
where  $\widehat{B}_{2^k} = B_{2^k}$ for $k\ge 0$ and $\widehat{B}_{2^{-1}} = \emptyset$. Thus, when $R=1$, \eqref{ioio} follows from \eqref{es2} and \eqref{opop}.

For general $R>0$, we take
\[\Tilde{u}(t,x) = R^{-\sigma}u(R^{\sigma/\alpha}t,Rx) \qx{and} \Tilde{f}(t,x) = f(R^{\sigma/\alpha}t,Rx) . \stepcounter{equation}\tag{\theequation}\label{we}\]
It follows that
$$\partial_t^{\alpha}\Tilde{u} - \Tilde{L}\Tilde{u} + \lambda \Tilde{u}= \Tilde{f} \qx{in} (0,\Tilde{T})\times \rr^d,$$
where $\Tilde{T} =R^{-\sigma/\alpha}T$, and $\Tilde{L}$ is the operator with the kernel $R^{d+\sigma} K(R^{\sigma/\alpha }t,Ry)$ satisfying \eqref{ass1} with the same $\nu$ and $\Lambda$ as $K$, the kernel of $L$. Since \begin{align*}
&\big( | \partial_t^\alpha \Tilde{u} |^{p} \big)_{(0,\Tilde{T})\times B_{1/2}}^{1/{p}} + \big( |L^1 \Tilde{u} |^{p} \big)_{(0,\Tilde{T})\times B_{1/2}}^{1/{p}} + \lambda\big( | \Tilde{u} |^{p} \big)_{(0,\Tilde{T})\times B_{1/2}}^{1/{p}} \\
&\le N \big( | \Tilde{f}|^{p} \big)_{(0,\Tilde{T})\times B_{1}}^{1/{p}}+   N \sum_{k = 0}^{\infty} 2^{-k\sigma} \big( |\Tilde{u} |^p \big)_{(0,\Tilde{T})\times B_{2^k}}^{1/p},
\end{align*}
we arrive at \eqref{ioio} by a change of variables.
\end{proof}

\subsection{The level set argument}
In this subsection, we prove Theorem \ref{main} for general $p \in (1,\infty)$.
Recall the equation
\[\partial_t^{\alpha}u - Lu + \lambda u = f.\stepcounter{equation}\tag{\theequation}\label{mq} \]

We start with a decomposition of the solution.
\begin{proposition}\label{eees}
Let $\alpha \in (0,1]$, $\sigma\in(0,2)$, $T \in (0, \infty)$, $\lambda \ge 0$ and $p  \in (1, \infty)$.
Suppose that Theorem \ref{main} holds for this $p$, and {$u \in \hh_{p,0}^{\alpha,\sigma}(T)$} satisfies
Equation \eqref{mq}.
Then there exists $p_1 = p_1(d, \alpha,\sigma,p)\in (p,\infty]$
satisfying
\[
p_1 - p > \delta(d,\alpha,\sigma)>0 \stepcounter{equation}\tag{\theequation}\label{p1} ,
\]
and the following holds.
For any $(t_0,x_0) \in [0,T] \times \rr^d$, $R >0 $, and $S = \min\{0, t_0 - R^{\sigma/\alpha}\}$,
there exist
\[
w \in \hh_{p,0}^{\alpha,\sigma}((t_0-R^{\sigma/\alpha}, t_0)\times \rr^d) \qx{and} v \in \hh_{p,0}^{\alpha,\sigma}((S,t_0) \times \rr^d),
\]
such that $u = w + v$ in $Q_R(t_0,x_0)$,
\[
( |L^1 w|^p )_{Q_R(t_0,x_0)}^{1/p}  + ( |\lambda w|^p )_{Q_R(t_0,x_0)}^{1/p} \le N ( |f|^p )_{Q_{2R}(t_0,x_0)}^{1/p}, \stepcounter{equation}\tag{\theequation}\label{www}
\]
\begin{align*}
\big( |L^1 v  |^{p_1} \big)_{Q_{R/2}(t_0,x_0)}^{1/p_1} &\leq N \big( |f|^p \big)_{Q_{2R}(t_0,x_0)}^{1/p} + N \sum_{k = 0}^{\infty} 2^{-k\sigma} \big( |f|^p \big)_{(t_0 -R^{\sigma/\alpha},t_0)\times B_{2^kR}(x_0)}^{1/p} \\
&\quad+  N \sum_{k = 0}^{\infty} 2^{-k\sigma} \big( |L^1 u|^p \big)_{(t_0 -R^{\sigma/\alpha},t_0)\times B_{2^kR}(x_0)}^{1/p}\\
&\quad+  N \sum_{k = 0}^{\infty} 2^{-k\alpha} \big( |L^1 u|^p \big)_{(t_0 - (2^{k+1}+1)R^{\sigma/\alpha},t_0)\times B_{R}(x_0)}^{1/p},
\stepcounter{equation}\tag{\theequation}\label{vvv}
\end{align*}
and
\begin{align*}
\big( |\lambda v  |^{p_1} \big)_{Q_{R/2}(t_0,x_0)}^{1/p_1} &\leq N \big( |f|^p \big)_{Q_{2R}(t_0,x_0)}^{1/p} + N \sum_{k = 0}^{\infty} 2^{-k\sigma} \big( |f|^p \big)_{(t_0 -R^{\sigma/\alpha},t_0)\times B_{2^kR}(x_0)}^{1/p} \\
&\quad+  N \sum_{k = 0}^{\infty} 2^{-k\sigma} \big( |\lambda u|^p \big)_{(t_0 -R^{\sigma/\alpha},t_0)\times B_{2^kR}(x_0)}^{1/p}\\
&\quad+  N \sum_{k = 0}^{\infty} 2^{-k\alpha} \big( |\lambda u|^p \big)_{(t_0 - (2^{k+1}+1)R^{\sigma/\alpha},t_0)\times B_{R}(x_0)}^{1/p},
\stepcounter{equation}\tag{\theequation}\label{vvvt}
\end{align*}
where $N = N(d,\nu,\Lambda,\alpha,\sigma,p)$,
$$
\big( |\,\cdot\, |^{p_1} \big)_{Q_{R/2}(t_0,x_0)}^{1/p_1} := \norm{{\,\cdot\,} }_{L_{\infty}(Q_{R/2}(t_0,x_0))} \qx{when} p_1 = \infty,
$$
and $u,f$ are extended to be zero for $t < 0$.
\end{proposition}
\begin{proof}
The proof of \eqref{vvvt} would be similar to the proof of \eqref{vvv} by replacing $L^1v$ with $\lambda v$. Thus, we focus on \eqref{vvv}.

By a shift of {the coordinates} and a scaling similar to \eqref{we}, without loss of generality, we assume that $x_0 = 0$ and $R = 1$.

For any $t_0 \in [0,T]$, we take a cutoff function $\zeta \in C_0^{\infty}((t_0 - 2^{\sigma/\alpha},t_0 +  2^{\sigma/\alpha})\times B_2)$ satisfying
\[ \zeta\in[0,1], \qx{and } \zeta = 1 \qx{in} (t_0-1,t_0)\times B_1. \]
By Theorem \ref{main}, there exists $w\in \hh^{\alpha,\sigma}_{p,0}((t_0-1,t_0)\times \rr^d)$ satisfying
\[
\partial_t^\alpha w - L w + \lambda w = \zeta f \qx{in} (t_0-1,t_0)\times \rr^d,
\]
and
\[
\norm{|L^1 w| + |\lambda w|}_{L_{p}((t_0-1,t_0)\times \rr^d)}
\leq N \norm{\zeta f}_{L_{p}((t_0-1,t_0)\times \rr^d)} \le N \norm{ f}_{L_{p}(Q_2(t_0,0))}, \stepcounter{equation}\tag{\theequation}\label{ww1}
\]
where $N = N(d, \nu, \Lambda, \alpha,\sigma, p)$. We then obtain the estimate \eqref{www}.

Next, we take $S = \min(0,t_0-1)$ and $v = u-w$. Indeed, by taking the zero extension of $w$ {for $t<t_0-1$}, we have $w\in \hh^{\alpha,\sigma}_{p,0}((S,t_0)\times \rr^d)$. See \cite[Lemma 3.5]{dong19} for details. Thus, it follows that $v\in \hh^{\alpha,\sigma}_{p,0}((S,t_0)\times \rr^d)$
and
\[
\partial_t^\alpha v - Lv + \lambda v= (1 - \zeta) f \qx{in} (S,t_0)\times \rr^d.
\]
Furthermore, we take $\eta\in C^\infty(\rr)$ such that
\[
 \eta(t) = \begin{cases}
1 \quad \quad \text{when}\,\,  t \in  (t_0-(1/2)^{\sigma/\alpha},t_0), \\
0 \quad \quad \text{when}\,\,  t \in (t_0-1,t_0+1)^c,
\end{cases}
{\qx{and } |\eta'| \le N(\alpha,\sigma).\stepcounter{equation}\tag{\theequation}\label{cuto}}
\]
It follows from \cite[Lemma 3.6]{dong19} that $ h: = \eta v \in \hh^{\alpha,\sigma}_{p,0}((t_0-1,t_0)\times \rr^d)$ and
\[
\partial_t^\alpha h - Lh + \lambda h=  \begin{cases} g + \eta(1 - \zeta) f\qx{in} (t_0 -1,t_0)\times \rr^d \quad &\text{when} \,\, \alpha \in (0,1),\\
\eta' v{+ \eta(1 - \zeta) f}  \qx{in} (t_0 -1,t_0)\times \rr^d \quad &\text{when}  \,\, \alpha = 1,
\end{cases}\stepcounter{equation}\tag{\theequation}\label{eqn2}
\]
where
$$
g(t,x) = \frac{\alpha}{\Gamma(1-\alpha)} \int_S^t (t-s)^{-\alpha-1}(\eta(t) - \eta(s)) v(s,x) \, ds.
$$
Moreover, we take
$$\xi \in C_0^\infty(\rr^d), \quad  \supp(\xi)\subset B_1,\quad\int_{\rr^d}\xi =1,\qx{and}\xi^\ep(\cdot) := \ep^{-d}\xi(\cdot/\ep).$$
To estimate $L^1 v$, we first mollify Equation \eqref{eqn2} and then take $L^1$ on both sides {to get}
\[
\partial_t^\alpha (L^1 h^\ep) - L ( L^1  h^\ep) + \lambda L^1  h^\ep= \Tilde{g}^\ep + L^1({\eta}[(1 - \zeta) f]^\ep)\qx{in} (t_0 -1,t_0)\times \rr^d ,
\stepcounter{equation}\tag{\theequation}\label{88889}
\]
where $v^\ep := v \ast_x \xi^\ep $
and
$$
\Tilde{g}^\ep(t,x) = \frac{\alpha}{\Gamma(1-\alpha)} \int_S^t (t-s)^{-\alpha-1}(\eta(t) - \eta(s)) L^1 v^\ep (s,x)\, ds,
$$
{when $\alpha\in (0,1)$ and $\tilde g(t,x)=\eta' Lv^{\varepsilon}$ when $\alpha=1$.}

Next, we take $p_1$ such that
 $$1/p_1= 1/p -\alpha\sigma/(2\alpha d +2 \sigma) \qx{if }p \le d/\sigma + 1/\alpha,$$ and  $$p_1 = \infty \qx{if }p > d/\sigma + 1/\alpha.$$
Note that $p_1$ satisfies \eqref{p1}, i.e., the increment is independent of $p$.
By the fact that $L^1  h^\ep \in \hh^{\alpha,\sigma}_{p,0}((t_0-1,t_0)\times \rr^d)$ and Corollary \ref{so}, we have
\begin{align*}
 \norm{ L^1  h^\ep }_{  L_{p_1} (Q_{1/2}(t_0, 0) ) } &\leq \norm{ L^1  h^\ep }_{  L_{p_1}( (t_0 - 1, t_0) \times B_{1/2} ) } \\
&\le N \norm{|L^1  h^\ep| + |\Tilde{g}^\ep|+ | L^1({\eta}[(1 - \zeta) f]^\ep)|}_{L_p( (t_0 - 1, t_0) \times B_{3/4} ) }\\
&\quad +  N \norm{L^1  h^\ep }_{L_p((t_0-1,t_0);L_1(\rr^d,\psi))},
\end{align*}
which, by taking the limit of $\ep \to 0$, implies that
\begin{align*}
\norm{L^1  h }_{ L_{p_1} (Q_{1/2}( t_0, 0 )) } &\le N \norm{|L^1  h| + |\Tilde{g}|+ | L^1({\eta}[(1 - \zeta) f])|}_{L_p( (t_0 - 1, t_0) \times B_{3/4} ) }\\
&\quad +  N \norm{L^1  h }_{L_p((t_0-1,t_0);L_1(\rr^d,\psi))},\stepcounter{equation}\tag{\theequation}\label{xx23}
\end{align*}
where
$$
\Tilde{g}(t,x) = \frac{\alpha}{\Gamma(1-\alpha)} \int_S^t (t-s)^{-\alpha-1}(\eta(t) - \eta(s)) L^1 v (s,x)\, ds
$$
when $\alpha\in (0,1)$ and $\tilde g(t,x)=\eta' Lv$ when $\alpha=1$.
Then, by the definition of $h$ and \eqref{xx23}, we obtain
\begin{align*}
& \norm{L^1 v}_{ L_{p_1} ( Q_{1/2} (t_0, 0) )   }
= \norm{ L^1  h}_{  L_{p_1} (Q_{1/2}(t_0, 0) ) } \\
&\le N \norm{|L^1  v| + |\Tilde{g}|}_{L_p( (t_0 - 1, t_0) \times B_{1} ) } +  N \norm{L^1((1 - \zeta) f)}_{L_p( (t_0 - 1, t_0) \times B_{3/4} ) }\\
&\quad +  N \norm{L^1 v }_{L_p((t_0-1,t_0);L_1(\rr^d,\psi))}
, \stepcounter{equation}\tag{\theequation}\label{xx1}
\end{align*}
where $N = N(d,\nu,\Lambda,\alpha,\sigma,p,p_1)$.
We estimate the terms on the right-hand side of \eqref{xx1} separately as follows.

When $\alpha \in (0,1)$, a similar computation as in \cite[Proposition 5.1]{dong19} leads to
\begin{align*}
\norm{\Tilde{g}}_{L_p( Q_1(t_0,0)  ) } &\le {N}\sum_{k=0}^\infty 2^{-k\alpha} \big( |L^1 u|^p \big)_{(t_0 - (2^{k+1}+1)^{\sigma/\alpha},t_0)\times B_{1}}^{1/p} \\
&\quad + N\sum_{k=0}^\infty 2^{-k\alpha}\big( |L^1 w|^p \big)_{(t_0 - (2^{k+1}+1)^{\sigma/\alpha},t_0)\times B_{1}}^{1/p}\\
&\le  {N}\sum_{k=0}^\infty 2^{-k\alpha} \big( |L^1 u|^p \big)_{(t_0 - (2^{k+1}+1)^{\sigma/\alpha},t_0)\times B_{1}}^{1/p} + N ( |L^1  w|^p )^{1/p}_{Q_1(t_0,0)}\\
&\le  {N}\sum_{k=0}^\infty 2^{-k\alpha} \big( |L^1 u|^p \big)_{(t_0 - (2^{k+1}+1)^{\sigma/\alpha},t_0)\times B_{1}}^{1/p} +  N ( |f|^p )_{Q_{2}(t_0,0)}^{1/p},\stepcounter{equation}\tag{\theequation}\label{xx2}
\end{align*}
where in the second inequality we used the fact that $w$ vanishes when $t<{t_0-1}$, and in the last inequality, we used \eqref{ww1}. {Recall} that when $\alpha = 1$, $\Tilde{g} = \eta' L^1 v$. Then it is easily seen that
$$
\norm{\Tilde{g}}_{L_p( Q_1(t_0,0)  ) } \le N \big( |L^1 u|^p \big)_{(t_0 - 3^{\sigma},t_0)\times B_{1}}^{1/p} +  N ( |f|^p )_{Q_{2}(t_0,0)}^{1/p}.
$$
Also, by the fact that $1-\zeta$ vanishes in $Q_1(t_0,0)$ together with the Minkowski inequality and H\"older's inequality, we have
\begin{align*}
&\norm{ L^1((1 - \zeta) f)}_{L_p( (t_0 - 1, t_0) \times B_{3/4} )   ) } \\
&\le N  \norm{ \int_{\rr^d} |(1 - \zeta(\cdot,\cdot+y)) f(\cdot,\cdot+y)| |y|^{-d-\sigma} \, dy}_{L_p( (t_0 - 1, t_0) \times B_{3/4} ) )}\\
&\le N \norm{ \int_{|y| > 1/4} |(1 - \zeta(\cdot,\cdot+y)) f(\cdot,\cdot+y)| |y|^{-d-\sigma} \, dy}_{L_p( (t_0 - 1, t_0) \times B_{3/4} ) )}\\
&\le N \norm{ \sum_{k = -1}^{\infty} 2 ^ {-kd-k\sigma}\int_{{B}_{2^k} \setminus {B}_{2^{k-1}}} | f(\cdot,\cdot+y)|  \, dy}_{L_p( (t_0 - 1, t_0) \times B_{3/4} ) )}\\
&\le N \sum_{k = 0}^{\infty} 2^{-k\sigma} \bigg(\fint_{t_0 -1}^{t_0} \fint_{B_{2^k}} |f (t,x)|^p \, dx \, dt \bigg)^{1/p}. \stepcounter{equation}\tag{\theequation}\label{xx4}
\end{align*}

Finally, using an estimate similar to \eqref{opop} together with the fact that $v = u-w$ and \eqref{ww1}, we obtain
\begin{align*}
&\norm{L^1 v}_{L_p((t_0-1,t_0);L_1(\rr^d,\psi))}\\
&\le N \sum_{k = 0}^{\infty} 2^{-k\sigma} \big( |L^1 v|^p \big)_{(t_0 -1,t_0)\times B_{2^k}}^{1/p}\\
&\le N \sum_{k = 0}^{\infty} 2^{-k\sigma} \big( \big( |L^1 u|^p \big)_{(t_0 -1,t_0)\times B_{2^k}}^{1/p}+ \big( |L^1 w|^p \big)_{(t_0 -1,t_0)\times B_{2^k}}^{1/p}\big)\\
&\le N \sum_{k = 0}^{\infty} 2^{-k\sigma} \big( |L^1 u|^p \big)_{(t_0 -1,t_0)\times B_{2^k}}^{1/p} +  N ( |f|^p )_{Q_{2}(t_0,0)}^{1/p}.\stepcounter{equation}\tag{\theequation}\label{xx3}
\end{align*}
Thus, combining \eqref{xx1}, \eqref{xx2}, \eqref{xx4}, and \eqref{xx3}, we arrive at \eqref{vvv} when $R=1$. {The proof of \eqref{vvvt} is similar (and actually simpler) by using
\begin{align*}
\norm{h}_{ L_{p_1} (Q_{1/2}( t_0, 0 )) } &\le N \norm{|h| + |g|}_{L_p( (t_0 - 1, t_0) \times B_{3/4} ) } +  N \norm{h }_{L_p((t_0-1,t_0);L_1(\rr^d,\psi))}
\end{align*}
instead of \eqref{xx23}. We omit the details.}
The proposition is proved.
\end{proof}

Next, with the estimates above, we verify Assumption \ref{ass44} of Lemma \ref{a6}, which is the key to the level set argument.
For $R> 0$, $p\in (1,\infty)$, and $p_1=p_1(d,\alpha,\sigma,p)$ from Proposition \ref{eees},
denote
\begin{align*}
    \aa(s) = \left\{ (t,x) \in (-\infty,T) \times \rr^d: |L^1 u(t,x)| > s \right\},
\end{align*}
and
\begin{align*}
    \cbb_{\gamma}(s) = \big\{ &(t,x) \in (-\infty,T) \times \rr^d:\\
&\gamma^{-1/p}( \cmm |f|^p (t,x) )^{1/p} + \gamma^{-1/p_1}( \css\cmm |L^1 u|^p(t,x))^{1/p} > s  \big\}. \stepcounter{equation}\tag{\theequation}\label{qq4}
\end{align*}
Also, we write
\begin{align*}
\cC_R(t,x) = (t-R^{\sigma/\alpha},t+R^{\sigma/\alpha}) \times B_R(x)\qx{and}
\widehat \cC_R(t,x)=\cC_R(t,x)\cap \{t\le T\}. \stepcounter{equation}\tag{\theequation}\label{qq}
\end{align*}
\begin{lemma}\label{3.3}
Let $\gamma \in (0,1)$, $\alpha \in (0,1]$, $\sigma\in(0,2)$, $T \in (0, \infty)$, $R> 0 $, $\lambda\ge 0$ and $p  \in (1, \infty)$.
Suppose that Theorem \ref{main} holds for this $p$, and {$u \in \hh_{p,0}^{\alpha,\sigma}(T)$} satisfies
Equation \eqref{mq}.
Then, there exists a {sufficiently large} constant $\kappa = \kappa(d,\nu,\Lambda,\alpha,\sigma,p) > 1$ such that for {any} $(t_0,x_0) \in (-\infty,T] \times \rr^d$ and $s>0$, if
\[
|\cC_{\Tilde{R}}(t_0,x_0) \cap \aa(\kappa s)|  \geq \gamma |\cC_{\Tilde{R}}(t_0,x_0)|,\stepcounter{equation}\tag{\theequation}\label{mm1}
\]
then we have
\[
\widehat\cC_{\Tilde{R}}(t_0,x_0) \subset \cbb_{\gamma}(s),
\]
where $\Tilde{R} = 2^{-1-\alpha/\sigma}R$.
\end{lemma}

\begin{proof}
By dividing Equation {\eqref{mq}} by $s$, we assume that $s=1$.

First note that since $u$ vanishes when $t<0$, if $t_0 + \Tilde{R}^{\sigma/\alpha}< 0$, then
$$
\cC_{\Tilde{R}}(t_0,x_0) \cap \aa(\kappa) \subset \left\{ (t,x) \in (-\infty,0) \times \rr^d: |{L^1 u(t,x)}| > 1 \right\} = \emptyset,
$$
and \eqref{mm1} is not satisfied. Thus, it suffices to consider {the case when} $t_0 +  \Tilde{R}^{\sigma/\alpha}\geq 0$.

We argue by contradiction. Suppose that there exist some $(s,y) \in \widehat\cC_{\Tilde{R}}(t_0,x_0)$
such that
\[
\gamma^{-1/p}( \cmm |f|^p (s,y) )^{1/p} + \gamma^{-1/p_1} ( \css\cmm |L^1 u|^p(s,y))^{1/p} \leq 1. \stepcounter{equation}\tag{\theequation}\label{zz1}
\]
Let $t_1 := \min \{ t_0 + \Tilde{R}^{\sigma/\alpha}, T\}$. Then by Proposition \ref{eees}, for $S=\min\{0,t_1-R^{\sigma/\alpha}\}$, there exist $w \in \hh_{p,0}^{\alpha,\sigma}((t_1-R^{\sigma/\alpha}, t_1) \times \rr^d)$ and $v \in \hh_{p,0}^{\alpha,\sigma}((S, t_1) \times \rr^d)$ such that $u = w + v$ in $Q_R(t_1,x_0)$,
\begin{align*}
( |L^1 w|^p )_{Q_R(t_1,x_0)}^{1/p} \le N ( |f|^p )_{Q_{2R}(t_1,x_0)}^{1/p}, \stepcounter{equation}\tag{\theequation}\label{zz2}
\end{align*}
and
\begin{align*}
\big( |L^1 v  &|^{p_1} \big)_{Q_{R/2}(t_1,x_0)}^{1/p_1}
\leq N \big( |f|^p \big)_{Q_{2R}(t_1,x_0)}^{1/p} \\
&+  N \sum_{k = 0}^{\infty} 2^{-k\sigma} \bigg(\fint_{t_1 -R^{\sigma/\alpha}}^{t_1} \fint_{B_{2^kR}(x_0)} |f (t,x)|^p \, dx \, dt \bigg)^{1/p}\\
&+ N \sum_{k=0}^\infty 2^{-k\alpha} \bigg( \fint^{t_1}_{t_1 - (2^{k+1}+1)R^{\sigma/\alpha}} \fint_{B_R(x_0)} |L^1 u(t,y)|^p \, dy \,dt \bigg)^{1/p}\\
& +  N \sum_{k = 0}^{\infty} 2^{-k\sigma} \bigg(\fint_{t_1 -R^{\sigma/\alpha}}^{t_1} \fint_{B_{2^kR}(x_0)} | L^1 u(t,y)|^p \, dx \, dt \bigg)^{1/p}\stepcounter{equation}\tag{\theequation}\label{zz3},
\end{align*}
where $N = N(d,\nu,\Lambda,\alpha,\sigma,p)$.
We have
$$
(s,y) \in \widehat\cC_{\Tilde{R}}(t_0,x_0)
\subset Q_{R/2}(t_1,x_0) \subset Q_{2R}(t_1,x_0),
$$
$$
(s,y) \in \widehat\cC_{\Tilde{R}}(t_0,x_0) \subset (t_1- (2^{k+1}+1)R^{\sigma/\alpha}, t_1) \times B_R(x_0),
$$
and
$$
(s,y) \in \widehat\cC_{\Tilde{R}}(t_0,x_0) \subset (t_1- R^{\sigma/\alpha}, t_1) \times B_{2^kR}(x_0)
$$
for all $k = 0,1,\ldots$.
By these set inclusions together with \eqref{zz1}, \eqref{zz2}, and \eqref{zz3}, we infer
$$
( |L^1 v|^{p_1} )_{Q_{R/2}(t_1,x_0)}^{1/p_1}  \leq N \gamma^{1/p_1} \qx{and} (|L^1 w|^p)^{1/p}_{Q_R(t_1,x_0)} \leq N\gamma^{1/p},
$$
where $N = N(d,\nu,\Lambda,\alpha,\sigma,p)$.
Then for a constant $C_1> 0 $ to be determined, by the Chebyshev inequality,
\begin{align*}
&|\cC_{\Tilde{R}}(t_0,x_0) \cap \aa(\kappa)|\\
&= |\{(t,x) \in \cC_{\Tilde{R}}(t_0,x_0), t \in (-\infty,T): |L^1 u(t,x)| > \kappa\}|\\
&\leq \left|\{ (t,x) \in Q_{R/2}(t_1,x_0): |L^1 u(t,x)| > \kappa\}\right|\\
&\leq \left|\{(t,x) \in Q_{R/2}(t_1,x_0): |L^1 w(t,x)| > \kappa - C_1 \}\right|\\
&\quad + \left|\{(t,x) \in Q_{R/2}(t_1,x_0): |L^1 v(t,x)| > C_1 \}\right|\\
&\leq\int_{Q_{R/2}(t_1,x_0)} (\kappa-C_1)^{-p}|L^1 w|^p +  C_1^{-p_1} |L^1 v|^{p_1}\, dx \, dt \\
&\leq \frac{N^p\gamma |Q_{R/2}|}{(\kappa - C_1)^p} + \frac{N^{p_1}\gamma|Q_{R/2}|}{C_1^{p_1}}1_{p_1 \neq \infty}\\
&\leq N_0(d,\alpha, \sigma) |\cC_{\Tilde{R}}(t_0,x_0)|\gamma \bigg(\frac{N^p  }{(\kappa - C_1)^p} + \Big(\frac{N}{C_1}\Big)^p 1_{p_1 \neq \infty} \bigg).
\end{align*}
By first taking a sufficiently large $C_1 = C_1(d,\nu,\Lambda,\alpha,\sigma,p)$ such that
$$
N_0  (N/C_1)^p < 1/2,
$$
and then taking {a large} $\kappa = \kappa(d,\nu,\Lambda,\alpha,\sigma,p)$ so that
$$
N_0 N^p/(\kappa-C_1)^p < 1/2,
$$
we obtain
\begin{align*}
|\cC_{\Tilde{R}}(t_0,x_0) \cap \aa(\kappa)|
 < \gamma |\cC_{\Tilde{R}}(t_0,x_0)|. \stepcounter{equation}\tag{\theequation}\label{q8}
\end{align*}
However, \eqref{q8} contradicts to \eqref{mm1}.
The lemma is proved.
\end{proof}
\begin{remark}\label{lam}
By replacing $L^1u$ with $\lambda u$ in the above proof together with \eqref{vvvt}, we conclude that Lemma \ref{3.3} holds for
\begin{align*}
    \aa'(s) = \left\{ (t,x) \in (-\infty,T) \times \rr^d: |\lambda u(t,x)| > s \right\},
\end{align*}
and
\begin{align*}
    \cbb'_{\gamma}(s) = \big\{ &(t,x) \in (-\infty,T) \times \rr^d:\\
&\gamma^{-1/p}( \cmm |f|^p (t,x) )^{1/p} + \gamma^{-1/p_1}( \css\cmm |\lambda u|^p(t,x))^{1/p} > s  \big\}.
\end{align*}
\end{remark}

We are ready to prove Theorem \ref{main} for general $p \in (1,\infty)$.
\begin{proof}[Proof of Theorem \ref{main}]
Since the existence of solutions for the equation
$$
\partial_t^\alpha u + \pf u = f
$$
was derived in \cite[Theorem 2.8]{kim20}, by the method of continuity, it suffices to prove the a priori estimate \eqref{ess}. Note that by the argument in Remark \ref{2.3}, if \eqref{ess} holds for smooth functions, then \eqref{up} holds for smooth functions. Thus, by the density of smooth functions in $\hh_{p,0}^{\alpha,\sigma}(T)$, without loss of generality, we assume that $u \in C_0^\infty([0,T] \times \rr^d)$ with $u(0,\cdot) = 0$.

We first consider the case when $p\in [2,\infty)$ by using an iterative argument to successively increase the exponent $p$, which is referred to as the bootstrap argument. Recall that we have proved the base case when $p =2$. Now we assume that the theorem holds for some $p_0\in [2,\infty)$, and we prove \eqref{ess} for $p\in (p_0,p_1)$, where $p_1= p_1(d,\alpha,\sigma,p_0)$ is from Proposition \ref{eees}.
Note that
\begin{align*}
    \|L^1 u\|_{L_p(\rr^d_T)}^p = p \int_0^\infty |\aa(s)| s^{p-1} \, ds = p
\kappa^p \int_0^\infty |\aa(\kappa s)| s^{p-1} \, ds,
\end{align*}
and Lemma \ref{3.3} together with Lemma \ref{a6} leads to
\[
|\aa(\kappa s)| \leq N(d,\alpha) \gamma|\cbb_{\gamma}(s)|
\]
for all $s \in (0,\infty)$, where $\kappa = \kappa(d,\nu,\Lambda,\alpha,\sigma,p)$ is from Lemma \ref{3.3} and $\aa, \cbb_{\gamma}$ are defined in \eqref{qq4}.
Thus, by the Hardy-Littlewood theorem {for strong maximal functions},
\begin{align*}
&\|L^1 u\|_{L_p(\rr^d_T)}^p
\leq N p  \kappa^p \gamma \int_0^\infty |\cbb_{\gamma}(s)| s^{p-1} \, ds\\
&\le N\gamma \int_0^\infty\left|\left\{ (t,x) \in (-\infty,T) \times \rr^d:\gamma^{-\frac 1{ p_1}}( \css\cmm |L^1 u|^{p_0}(t,x))^{\frac 1 {p_0}} > s/2 \right\}\right| s^{p-1} \, ds\\
&\quad +
N\gamma \int_0^\infty\left|\left\{ (t,x) \in (-\infty,T) \times \rr^d:\gamma^{-\frac 1 {p_0}}( \cmm |f|^{p_0} (t,x) )^{\frac 1 {p_0}} > s/2 \right\}\right| s^{p-1} \, ds\\
&\leq N \gamma^{1-p/p_1} \|L^1 u\|^p_{L_p(\rr^d_T)} + N \gamma^{1-p/p_0} \|f\|^p_{L_p(\rr^d_T)},
\end{align*}
where $N = N(d,\nu,\Lambda,\alpha,\sigma,p)$. Similarly, by Remark \eqref{lam}, we have
\[
|\aa'(\kappa s)| \leq N(d,\alpha) \gamma|\cbb'_{\gamma}(s)|,
\]
which implies that
\begin{align*}
&\|\lambda u\|_{L_p(\rr^d_T)}^p \leq N \gamma^{1-p/p_1} \|\lambda u\|^p_{L_p(\rr^d_T)} + N \gamma^{1-p/p_0} \|f\|^p_{L_p(\rr^d_T)}.
\end{align*}
By the fact that $p< p_1$, we take a sufficiently small $\gamma \in (0,1)$ so that
$$
N \gamma^{1-p/p_1} < 1/2,
$$
which implies \eqref{ess} for $p \in (p_0,p_1)$.
We repeat this procedure.
Recall that $p_1 - p $ depends only on $d, \alpha, \sigma$. Thus in finite steps, we get a $p_0$ which is larger than $d/2 + 1/\alpha$, so that $p_1=p_1(d,\alpha,p_0)=\infty$. Therefore, the theorem is proved for {any} $p\in [2,\infty)$.

For $p \in (1,2)$, we use a duality argument. Again, we assume that $u \in C_0^\infty([0,T] \times \rr^d)$ with $u(0,x) = 0$ and prove \eqref{ess}. Let $L^{\ast}$ be the operator with the kernel $K(-t,-y)$, where $K$ is the kernel of $L$. For $p' = p/(p-1)$ and $\phi \in L_{p'}(\rr^d_T)$, there exist $w \in \hh_{p',0}^{\alpha,\sigma}((-T,0) \times \rr^d)$ satisfying
$$
\partial_t^\alpha w - L^{\ast} w + \lambda w = \phi(-t,x) \qx{in} (-T,0) \times \rr^d
$$
and
$$
\||L^1 w| + |\lambda w|\|_{L_{p'}((-T,0) \times \rr^d)} \leq N \|\phi(-t,x)\|_{L_{p'}((-T,0) \times \rr^d)} = N \|\phi\|_{L_{p'}(\rr^d_T)},
$$
where $\partial_t^\alpha w = \partial_t I_{-T}^{1-\alpha} w$ and $N =N = N(d, \nu, \Lambda, \alpha,\sigma, p)$.
It follows that
\begin{align*}
&\int_0^T \int_{\rr^d} \phi L^1 u  \, dx \, dt = \int_{-T}^0 \int_{\rr^d} \phi(-t,x)L^1 u(-t,x) \, dx \, dt\\
&= \int_{-T}^0 \int_{\rr^d} (\partial_t^\alpha w(t,x) - L^{\ast} w (t,x)) L^1 u(-t,x) \, dx \, dt\\
&= \int_0^T \int_{\rr^d} ( \partial_t^\alpha u (t,x)- Lu  (t,x)) (L^1)^{{*}} w(-t,x) \, dx \, dt\\
&= \int_0^T \int_{\rr^d} f(t,x)(L^1)^{{*}} w(-t,x) \, dx \, dt \leq N\|f\|_{L_p(\rr^d_T)} \|\phi\|_{L_{p'}(\rr^d_T)}.
\stepcounter{equation}\tag{\theequation}\label{zzz2}\end{align*}
Note that for the third equality, if $w$ is smooth, then we apply the Plancherel theorem to the integral of $ L^{\ast} w{(t,x)}L^1 u{(-t,x)}$ and apply an integration by parts to {the integral of} $\partial_t^\alpha w{(t,x)} L^1 u{(-t,x)}$ using the zero initial condition{s} of $w$ and $u$.  If $w$ is not necessarily smooth, we take $w_k \in C_0^\infty([-T,0] \times \rr^d)$ with $w_k(-T,0) = 0$ such that
$$
w_k \to w\quad \text{in}\,\, \hh_{{p'},0}^{\alpha,\sigma}((-T,0) \times \rr^d ).
$$
Similarly, we have
\begin{align*}
\int_0^T \int_{\rr^d} \phi \lambda u  \, dx \, dt
= \int_0^T \int_{\rr^d} f(t,x)\lambda w(-t,x) \, dx \, dt \leq N\|f\|_{L_p(\rr^d_T)} \|\phi\|_{L_{p'}(\rr^d_T)}.
\stepcounter{equation}\tag{\theequation}\label{zzz3}\end{align*}
Thus, by \eqref{zzz2} and \eqref{zzz3}, we arrive at \eqref{ess}. {As before, by the method of continuity, the solvability and thus theorem are proved for the remaining case when $p\in (1,2)$.}
\end{proof}

\section{Equations in \texorpdfstring{$L_{p,q,w}$ when $\alpha = 1$}p}\label{4}
In this section, we prove Theorem \ref{main2} by deriving a mean oscillation estimate of $u \in \hh_{p_0}^{1, \sigma} ( (-\infty,T)\times \rr^d )$ satisfying
\[
\partial_t u - L u + \lambda u= f
\quad \text{in}\,\,  (-\infty,T)\times \rr^d.
\]
In particular, for $t_0 \in (-\infty, T]$, we {are going to} decompose
$u =w + v$ in $(t_0-1,t_0) \times \rr^d$ and estimate them separately,
where $w \in \hh_{p_0, 0}^{1, \sigma} ( (t_0-1,t_0) \times \rr^d ) $ satisfies
\[
\partial_t w - L w + \lambda w=  f
\quad \text{in}\,\,  (t_0-1,t_0) \times \rr^d \stepcounter{equation}\tag{\theequation}\label{9998}
\]
and
\begin{align*}
\partial_t v - L v + \lambda v = 0 \quad \quad \text{in}\,\,  (t_0-1,t_0) \times \rr^d.
\stepcounter{equation}\tag{\theequation}\label{9999}
\end{align*}
Note that the decomposition above differs from the one in the previous section.

Throughout this section, We denote $L^1$ to be an operator satisfying \eqref{sig1} {when $\sigma=1$} and \eqref{ass}. Furthermore, for $r>0$ and $z=(z^1,\ldots,z^d)\in \rr^d$, the cube center{ed} at $z$ {with radius $r$} is defined as $$C_r(z):= (z^1-r/2,z^1+r/2) \times (z^2-r/2,z^2+r/2) \times \ldots \times (z^d-r/2,z^d+r/2).$$

We start with the estimate of $v$. In the following proposition, using the Sobolev embedding and an iteration argument, we bound the H\"older semi-norm of $L^1 v$, in particular $\pf v$, by a sum of its $L_{p_0}$ norms taken over cubes of fixed size centered at integer points. {For the convenience of computation, we use cubes instead of balls in the spatial coordinates. Note that since ${B_{r/2}\subset} C_r \subset B_{\sqrt{d}r/2}$, we can replace the norms taken over balls with norms taken over cubes in Corollary \ref{so}.}

\begin{proposition}\label{ho}
Let $\sigma\in(0,2)$, $p_0  \in (1, \infty)$, $t_0\in \rr$, and $v \in  \hh_{p_0}^{1, \sigma} ( (-\infty,t_0) \times \rr^d )$ satisfy \eqref{9999}. Then for any $r>0$,
\begin{enumerate}
\item for any $p\in (p_0,\infty)$ and $x_0\in \rr^d$,
we have
\begin{align*}
( | L^1 v|^{p} )^{1 /p}_{Q_{r/2} (t_0, x_0) }
\le     N \sum_{z \in \zz^d}  (1 + |z|^{d+\sigma})^{-1}\big( |L^1 v|^{p_0} \big)_{(t_0 -r^\sigma,t_0)\times C_r(z+x_0)}^{1/{p_0}}\stepcounter{equation} \tag{\theequation}\label{homo2}
\end{align*}
and
\begin{align*}
( | \lambda v|^{p} )^{1 /p}_{Q_{r/2} (t_0, x_0) }
\le     N \sum_{z \in \zz^d}  (1 + |z|^{d+\sigma})^{-1}\big( |\lambda v|^{p_0} \big)_{(t_0 -r^\sigma,t_0)\times C_r(z+x_0)}^{1/{p_0}},\stepcounter{equation} \tag{\theequation}\label{homoo2}
\end{align*}
where $N = N(d,\nu,\Lambda,\sigma,p_0,p)$.
\item For any $x_0\in \rr^d$, there exists $\tau = \tau (d,\sigma )\in(0,1)$ such that
\begin{align*}
[L^1 v]_{C^{\tau/\sigma, \tau} ( Q_{r/2} (t_0, x_0) ) }\leq  N r^{-\tau} \sum_{z \in \zz^d}  (1 + |z|^{d+\sigma})^{-1}\big( |L^1 v|^{p_0} \big)_{(t_0 -r^\sigma,t_0)\times C_r(z+x_0)}^{1/{p_0}}\stepcounter{equation} \tag{\theequation}\label{homo}
\end{align*}
and
\begin{align*}
[\lambda v]_{C^{\tau/\sigma, \tau} ( Q_{r/2} (t_0, x_0) ) }\leq  N r^{-\tau} \sum_{z \in \zz^d}  (1 + |z|^{d+\sigma})^{-1}\big( |\lambda v|^{p_0} \big)_{(t_0 -r^\sigma,t_0)\times C_r(z+x_0)}^{1/{p_0}}\stepcounter{equation} \tag{\theequation}\label{homoo},
\end{align*}
where $N = N(d,\nu,\Lambda,\sigma,p_0)$.
\end{enumerate}
\end{proposition}

\begin{proof}
The proofs of \eqref{homoo2} and \eqref{homoo} are similar to the proofs of \eqref{homo2} and \eqref{homo} by replacing $L^1v$ with $\lambda v$. Thus, we focus on \eqref{homo2} and \eqref{homo}.

By scaling and shifting the coordinates, we assume that $r=1$ and $x_0 = 0$.
Moreover, for $m = 1,2,\ldots$, we take
$p_{m} = p_{m}(d,\sigma,p_0)$
satisfying
\[
1/p_m= 1/p_0 -m\sigma/( d + \sigma) \stepcounter{equation} \tag{\theequation}\label{ip2}
\]
whenever the right-hand side is positive.

We first prove \eqref{homo2}. We take cutoff function $\eta\in C^\infty(\rr)$ satisfying \eqref{cuto}.
It follows that $\eta v \in \hh^{1,\sigma}_{{p_0},0}((t_0-1,t_0)\times \rr^d)$ and
\[
\partial_t( \eta v) - L(\eta v) + \lambda (\eta v) = \partial_t( \eta v) -\eta \partial_t  v = v  \eta' \qx{in} (t_0 -1,t_0)\times \rr^d .\stepcounter{equation} \tag{\theequation}\label{y6}
\]
Taking $L^1$ on both sides of \eqref{y6} leads to
\[
\partial_t (\eta L^1 v) - L(\eta L^1 v) + \lambda (\eta L^1 v) = L^1 v \eta' \qx{in}(t_0 -1,t_0)\times \rr^d.\stepcounter{equation} \tag{\theequation}\label{y7}
\]
Note that if $v$ is not regular enough, we can first mollify the equation and then take the limit of the estimate as in \eqref{88889}.
Moreover, note that by the Minkowski inequality
\begin{align*}
&\norm{L^1 v}_{L_{p_0}((t_0-1,t_0);L_1(\rr^d,\psi))}\\
&=\bigg( \int_{t_0-1}^{t_0} \big(\int_{\rr^d} |L^1 v(t,x)| \frac{1}{1+ |x|^{d+\sigma}}\,dx \big)^{p_0}\,dt \bigg)^{1/p_0}\\
&\le\int_{\rr^d} \big(\int_{t_0-1}^{t_0} |L^1 v(t,x)|^{p_0} \,dt\big)^{1/p_0} \frac{1}{1+ |x|^{d+\sigma}}\,dx  \\
&\le \sum_{z \in \zz^d} \int_{C_1(z)} \big(\int_{t_0-1}^{t_0} |L^1 v(t,x)|^{p_0} \,dt\big)^{1/p_0} \frac{1}{1+ |x|^{d+\sigma}}\,dx  \\
&\le N(d,\sigma,{p_0})\sum_{z\in \zz^d} (1 + |z|^{d+\sigma})^{-1} \bigg( \fint_{t_0 -1}^{t_0} \fint_{C_1(z)} | L^1 v(t,x)|^{p_0} \, dx \, dt\bigg)^{1/{p_0}} \\
&\le N \sum_{z \in \zz^d}  (1 + |z|^{d+\sigma})^{-1}\big( |L^1 v|^{p_0} \big)_{(t_0 -1,t_0)\times C_1(z)}^{1/{p_0}}.
\end{align*}
Hence, by applying Corollary \ref{so} to \eqref{y7}, we have
\begin{align*}
&\norm{L^1 v}_{ L_{p_1} ( Q_{1/2} (t_0, 0) )   }\\
&\leq \norm{\eta L^1  v}_{  L_{p_1}( (t_0 - 1, t_0) \times B_{1/2} ) } \\
&\le N \norm{L^1  v}_{L_{p_0}( (t_0 - 1, t_0) \times B_{1} ) }+  N \norm{L^1 v }_{L_{p_0}((t_0-1,t_0);L_1(\rr^d,\psi))}\\
&\le  N \sum_{z \in \zz^d}  (1 + |z|^{d+\sigma})^{-1}\big( |L^1 v|^{p_0} \big)_{(t_0 -1,t_0)\times C_1(z)}^{1/{p_0}},
\stepcounter{equation}\tag{\theequation}\label{8888}
\end{align*}
where $N = N(d,\nu,\Lambda,\sigma,p_0)$. If $p\le p_1$, then we arrive at \eqref{homo2} by H\"older's inequality.

Otherwise, note that with a minor modification of the above {estimates}, we conclude that $$v,\, \pf v,\, \partial_t v \in L_{p_1} ((t_0-1/2^\sigma,t_0) \times  B_{3/4}).$$
Furthermore, we take a cutoff function
$$\zeta \in C_0^\infty(\rr^d), \quad \zeta = 1 \qx{in} B_{1/2}, \qx{and }\zeta = 0 \qx{in} B^c_{3/4}.$$ By taking the mollification (in both the time and the spatial variables) of $\zeta v$ as the defining sequence, we conclude that
\[\zeta v \in \hh_{p_1}^{1, \sigma} ( (t_0-1/2^\sigma,t_0) \times \rr^d ).\stepcounter{equation}\tag{\theequation}\label{il2}\]
Moreover, similar to \eqref{8888} with a shift {of the coordinates} and some minor modifications, we have
\begin{align*}
& \norm{L^1 v}_{ L_{p_1} ( (t_0 -(1/2)^\sigma,t_0)\times B_{\sqrt{d}/2}(x_0) )   }\\
&\le  N \sum_{z \in \zz^d}  (1 + |z|^{d+\sigma})^{-1}\big( |L^1 v|^{p_0} \big)_{(t_0 -1,t_0)\times C_1(z+x_0)}^{1/{p_0}}
\stepcounter{equation}\tag{\theequation}\label{il1}.
\end{align*}
Therefore, by taking anther cutoff function in time together with \eqref{il2}, Corollary \ref{so}, and \eqref{il1}, we obtain
\begin{align*}
& \norm{L^1 v}_{ L_{p_2} ( Q_{1/4} (t_0, 0) )   }\\
&\le  N \sum_{z \in \zz^d}  (1 + |z|^{d+\sigma})^{-1}\big( |L^1 v|^{p_1} \big)_{(t_0 -(1/2)^\sigma,t_0)\times C_{1/2}(z)}^{1/{p_1}}\\
&\le  N \sum_{z \in \zz^d}  (1 + |z|^{d+\sigma})^{-1} \sum_{y \in \zz^d} (1 + |y-z|^{d+\sigma})^{-1}\big( |L^1 v|^{p_0} \big)_{(t_0 -1,t_0)\times C_{1}(y)}^{1/{p_0}}\\
&\le  N \sum_{z \in \zz^d}  (1 + |z|^{d+\sigma})^{-1}\big( |L^1 v|^{p_0} \big)_{(t_0 -1,t_0)\times C_1(z)}^{1/{p_0}},
\end{align*}
where for the last inequality, we used
\begin{align*}
&\sum_{z \in \zz^d} (1 + |z|^{d+\sigma})^{-1}(1 + |y-z|^{d+\sigma})^{-1} \\
&\le N(d,\sigma)  (1 + |y|^{d+\sigma})^{-1}  \sum_{z \in \zz^d} \Big[(1 + |z|^{d+\sigma})^{-1} + (1 + |y-z|^{d+\sigma})^{-1}\Big]\\
&\le N(d,\sigma) (1 + |y|^{d+\sigma})^{-1} .
\end{align*}

We repeat the above process with $p_m$, $m = 3,4,5,\ldots$, for finite many times until $p_i \ge p$ or $p_i > d/\sigma + 1$. Note that the number of iterations depends only on $d,\sigma$, and $p$ by \eqref{ip2}. Thus, by a covering argument, we obtain \eqref{homo2}.

For \eqref{homo}, if $p_0 \ge 2(d/\sigma + 1)$, then by Corollary \ref{so}, we can replace the left-hand side of the inequality \eqref{8888} with $[L^1 v]_{C^{\tau /\sigma, \tau} ( Q_{1/2} (t_0, 0) ) }$, we get \eqref{homo}.
Otherwise, we apply the iteration argument as above until $p_k \ge 2(d/\sigma + 1)$, and by a covering argument, we arrive at \eqref{homo}.
The proposition is proved.
\end{proof}

\begin{remark}\label{657}
It is worth noting that the operator we considered in Proposition \ref{ho} was local in time, i.e., $\alpha =1$. For general $\alpha \neq 1$, it is not clear to us whether one can derive a similar estimate. In particular, we cannot get an expression as simple as \eqref{y7} (see the extra $\Tilde{g}$ term on \eqref{88889}), {and} the second  inequality of \eqref{8888} no longer holds.
\end{remark}

Next, we estimate the non-homogeneous part, which satisfies the zero initial condition  at $t=t_0-1$. In the following proposition, we denote $\norm{\, \cdot\,}_{p_0,\Omega} := \norm{\,\cdot\,}_{L_{p_0}((t_0-1,t_0) \times \Omega)}$ for any $\Omega\subset \rr^d $ and  $\norm{\,\cdot\,}_{p_0} := \norm{\,\cdot\,}_{L_{p_0}((t_0-1,t_0) \times \rr^d)}$.
\begin{proposition}\label{nonhomo}
Let $t_0\in \rr$ and {$w \in \hh_{p_0, 0}^{1, \sigma} ( (t_0-1,t_0) \times \rr^d ) $} {be} such that
\[
\partial_t w - L w + \lambda w = f \quad \text{in}\quad  (t_0-1,t_0) \times \rr^d .
\]
Then, for any $x_0 \in \rr^d$ and $R\ge1$, we have
\begin{align*}
&( |L^1 w|^{p_0} )^{1 / p_0}_{(t_0 -1 ,t_0)\times B_{R/2}(x_0)} + \lambda \big( |  w|^{p_0} \big)_{(t_0 -1 ,t_0)\times B_{R/2}}^{1/{p_0}}\\
&\le N \sum_{k = 0}^{\infty} 2^{-k\sigma} \big( |f|^{p_0} \big)_{(t_0 - 1,t_0)\times B_{2^kR}(x_0)}^{1/{p_0}},\stepcounter{equation}\tag{\theequation}\label{oo}
\end{align*}
where $N =N(d,\nu,\Lambda,\sigma,p_0)$.
\end{proposition}
\begin{proof}
By shifting the coordinates, we assume that $x_0 =0$. It follows from Corollary \ref{lr} that
\begin{align*}
 &\big( | L^1 w|^{p_0} \big)_{(t_0 -1 ,t_0)\times B_{R/2}}^{1/{p_0}} + \lambda \big( |  w|^{p_0} \big)_{(t_0 -1 ,t_0)\times B_{R/2}}^{1/{p_0}}\\
 &\le N \big( | f|^{p_0} \big)_{(t_0 -1 ,t_0)\times B_{R}}^{1/{p_0}}+   N \sum_{k = 0}^{\infty} 2^{-k\sigma} \big( |w|^{p_0} \big)_{(t_0 -1 ,t_0)\times B_{2^kR}}^{1/p_0}\\
 &:= N \big( | f|^{p_0} \big)_{(t_0 -1 ,t_0)\times B_{R}}^{1/{p_0}}+   N \sum_{k = 0}^{\infty} 2^{-k\sigma} A_k
 \stepcounter{equation}\tag{\theequation}\label{p01}.
\end{align*}
It remains to estimate $A_k$. By Lemma \ref{a2} and Corollary \ref{lr}, for each $k= 0,1,\ldots$,
\begin{align*}
A_k &\le N \big( |\partial_t w|^{p_0} \big)_{(t_0 -1 ,t_0)\times B_{2^kR}}^{1/p_0} \\
&\le N \big( | f|^{p_0} \big)_{(t_0 -1 ,t_0)\times B_{2^{k+1}R}}^{1/{p_0}}+   N \sum_{j = 0}^{\infty} 2^{-(j+k+1)\sigma} \big( |w|^{p_0} \big)_{(t_0 -1 ,t_0)\times B_{2^{j+k+1}R}}^{1/p_0}\\
&\le N \big( | f|^{p_0} \big)_{(t_0 -1 ,t_0)\times B_{2^{k+1}R}}^{1/{p_0}}+   N \sum_{j = k+1}^{\infty} 2^{-j\sigma} A_j,\stepcounter{equation}\tag{\theequation}\label{p02}
\end{align*}
where $N$ is independent of $k$.
By first multiplying both sides of \eqref{p02} by $2^{-\sigma k}$ and then summing over $k = k_0,k_0+1,\ldots$, for some integer $k_0$ to be determined, we obtain
\begin{align*}
&\sum_{k=k_0}^{\infty} 2^{-\sigma k} A_k \\
&\le  N \sum_{k=k_0}^{\infty }2^{-\sigma (k+1)}\big( | f|^{p_0} \big)_{(t_0 -1 ,t_0)\times B_{2^{k+1}R}}^{1/{p_0}} + N \sum_{k=k_0}^{\infty }2^{-\sigma k} \sum_{j=k+1}^{\infty}2^{-\sigma j }A_j\\
&\le   N \sum_{k=k_0}^{\infty }2^{-\sigma (k+1)}\big( | f|^{p_0} \big)_{(t_0 -1 ,t_0)\times B_{2^{k+1}R}}^{1/{p_0}}+ \frac{N 2^{-\sigma k_0}}{1 - 2^{-\sigma}}
 \sum_{j = k_0 + 1}^{\infty} 2^{-\sigma j} A_j.
\end{align*}
Picking $k_0$ sufficiently large so that $ \frac{N  2^{-\sigma k_0}}{1 - 2^{-\sigma}} \leq 1/2$, we have
\[
\sum_{k = k_0}^{\infty} 2^{-\sigma k} A_k
\leq  N \sum_{k = k_0}^\infty  2^{-\sigma k}
 ( |f|^{p_0} )^{1/p_0}_{(t_0-1, t_0 ) \times B_{2^{k}R}}.
\]
Therefore, by induction,
\[\sum_{k = 0}^{\infty} 2^{-\sigma k} A_k
\leq  N \sum_{k = 0}^\infty  2^{-\sigma k}
 ( |f|^{p_0} )^{1/p_0}_{(t_0-1, t_0 ) \times B_{2^{k}R}}. \stepcounter{equation}\tag{\theequation}\label{008}\]
Indeed, if there exists $N =N(d,\nu,\Lambda,\sigma,p_0)$ such that
$$
\sum_{k = j}^{\infty} 2^{-\sigma k} A_k
\leq  N \sum_{k = j}^\infty  2^{-\sigma k}
 ( |f|^{p_0} )^{1/p_0}_{(t_0-1, t_0 ) \times B_{2^{k}R}}
$$
for $j = 1,2,\ldots, k_0$, then, by \eqref{p02},
\begin{align*}
&\sum_{k = j-1}^{\infty} 2^{-\sigma k} A_k \le  2^{-\sigma(j-1) } A_{j-1} + \sum_{k = j}^{\infty} 2^{-\sigma k} A_k \\
&\le  N  2^{-\sigma j} \big( | f|^{p_0} \big)_{(t_0 -1 ,t_0)\times B_{2^{j}R}}^{1/{p_0}}+   N \sum_{k = j}^{\infty} 2^{-k\sigma} A_k \\
&\le N \sum_{k = j-1}^\infty  2^{-\sigma k}
 ( |f|^{p_0} )^{1/p_0}_{(t_0-1, t_0 ) \times B_{2^{k}R}}.
\end{align*}
Therefore, by \eqref{p01} and \eqref{008}, we arrive at \eqref{oo}, and the lemma is proved.
\end{proof}
\begin{remark}\label{658}
It is worth noting that although the operator we considered in Proposition \ref{nonhomo} was local in time, i.e. $\alpha =1$, {the proof still works} for any $\alpha \in (0,1]$ since the embedding used in \eqref{p02} still holds {in this case}.
\end{remark}

With the estimates of $v$ and $w$, we are ready to prove a mean oscillation estimate for $u$.
\begin{proposition}\label{4.5}
Let $\sigma\in(0,2)$, $T \in (0, \infty)$, $p_0  \in (1, \infty)$, and $u \in \hh_{p_0}^{1, \sigma} ( (-\infty,T)\times \rr^d )$ satisfy
\[
\partial_t u - L u + \lambda u = f
\quad \text{in}\,\,  (-\infty,T)\times \rr^d .
\]
Then, for any $(t_0, x_0) \in (-\infty,T] \times \rr^d$,  $r \in (0, \infty)$, and $\kappa \in (0, 1/4)$, we have
\begin{align*}
& ( |L^1 u - (L^1 u )_{{Q_{\kappa r}(t_0,x_0)}} | )_{{Q_{\kappa r}(t_0,x_0)}
} + ( |\lambda u - (\lambda u )_{{Q_{\kappa r}(t_0,x_0)}} | )_{{Q_{\kappa r}(t_0,x_0)}
}\\
&\quad \leq N \kappa^\tau \sum_{k=0}^{\infty} 2^{-k\sigma} \big[ \big( |L^1 u|^{p_0} \big)_{(t_0 -r^\sigma,t_0)\times B_{2^{k}r} (x_0)}^{1/{p_0}} + \big( |\lambda u|^{p_0} \big)_{(t_0 -r^\sigma,t_0)\times B_{2^{k}r} (x_0)}^{1/{p_0}} \big]\\
&\quad + N \kappa^{- (d+\sigma )/p_0}\sum_{k = 0}^{\infty}
2^{-k\sigma} \big( |f|^{p_0} \big)_{(t_0 - r^\sigma,t_0)\times B_{2^kr}(x_0)}^{1/{p_0}}, \stepcounter{equation} \tag{\theequation}\label{i3}
\end{align*}
where $\tau = \tau (d,\sigma, p_0)$ and $N = N(d,\nu,\Lambda,\sigma,p_0)$.
\end{proposition}
\begin{proof}
By shifting {the coordinates} and scaling, we assume $x_0 = 0$ and $r=1$.
Moreover, by Theorem \ref{main}, there exists {$w \in \hh_{p_0, 0}^{1, \sigma} ( (t_0-1,t_0) \times \rr^d)$} satisfying \eqref{9998}. Also, $v := u-w \in  \hh_{p_0}^{1, \sigma} ( (-\infty,t_0) \times \rr^d )$ satisfies \eqref{9999}.

Next, by H\"older's inequality and Proposition \ref{nonhomo},
\begin{align*}
& ( |L^1 u - (L^1 u )_{Q_{\kappa} (t_0, 0)} | )_{
Q_{\kappa} (t_0, 0) } + ( |\lambda u - (\lambda u )_{Q_{\kappa} (t_0, 0)} | )_{
Q_{\kappa} (t_0, 0) } \\
&\leq
( |L^1 v - (L^1 v )_{Q_{\kappa} (t_0, 0)} | )_{
Q_{\kappa} (t_0, 0)} + N \kappa^{- (d+\sigma)/p_0}
(| L^1 w|^{p_0} )^{1/p_0}_{Q_{1/2} (t_0, 0)}\\
& \quad + ( |\lambda v - (\lambda v )_{Q_{\kappa} (t_0, 0)} | )_{
Q_{\kappa} (t_0, 0)} + N \kappa^{- (d+\sigma)/p_0}
(| \lambda w|^{p_0} )^{1/p_0}_{Q_{1/2} (t_0, 0)}\\
&\leq
( |L^1 v - (L^1 v )_{Q_{\kappa} (t_0, 0)} | )_{
Q_{\kappa} (t_0, 0)} +  ( |\lambda v - (\lambda v )_{Q_{\kappa} (t_0, 0)} | )_{
Q_{\kappa} (t_0, 0)} \\
&\quad + N \kappa^{- (d+\sigma)/p_0}
 \sum_{k = 0}^{\infty} 2^{-k\sigma} \big( |f|^{p_0} \big)_{(t_0 - 1,t_0)\times B_{2^k}}^{1/{p_0}}\stepcounter{equation} \tag{\theequation}\label{e1}.
\end{align*}
For the first two terms on the right-hand side of \eqref{e1}, by Proposition \ref{ho},
there exists $\tau\in(0,1)$ such that
\begin{align*}
&(|L^1 v - (L^1 v)_{Q_\kappa ( t_0,0) } | )_{Q_\kappa ( t_0,0)}+ (|\lambda v - (\lambda v)_{Q_\kappa ( t_0,0) } | )_{Q_\kappa ( t_0,0)}\\
&\le \kappa^\tau [L^1 v]_{C^{\tau /\sigma, \tau} (Q_{1/4} (t_0, 0)  )  } +  \kappa^\tau [\lambda v]_{C^{\tau /\sigma, \tau} (Q_{1/4} (t_0, 0)  )  }\\
&\le N\kappa^\tau \sum_{z \in \zz^d}  (1 + |z|^{d+\sigma})^{-1} \big[ \big( |L^1 v|^{p_0} \big)_{(t_0 -1,t_0)\times C_1(z)}^{1/{p_0}}+ \big( |\lambda v|^{p_0} \big)_{(t_0 -1,t_0)\times C_1(z)}^{1/{p_0}} \big]\\
&\leq  N \kappa^\tau\sum_{z \in \zz^d}  (1 + |z|^{d+\sigma})^{-1}  \big( |L^1 u|^{p_0} \big)_{(t_0 -1,t_0)\times C_1(z)}^{1/{p_0}}\\
&\quad +   N \kappa^\tau\sum_{z \in \zz^d}  (1 + |z|^{d+\sigma})^{-1}\big( |\lambda u|^{p_0} \big)_{(t_0 -1,t_0)\times C_1(z)}^{1/{p_0}}\\
&\quad+  N \kappa^\tau \sum_{z \in \zz^d}  (1 + |z|^{d+\sigma})^{-1} \big[ \big( |L^1 w|^{p_0} \big)_{(t_0 -1,t_0)\times C_1(z)}^{1/{p_0}}+ \big( |\lambda w|^{p_0} \big)_{(t_0 -1,t_0)\times C_1(z)}^{1/{p_0}} \big]\\
&=: N\kappa^\tau I_1 + N\kappa^\tau I_2 +N \kappa^\tau I_3 . \stepcounter{equation} \tag{\theequation}\label{i2}
\end{align*}
To estimate $I_1,I_2$ and $I_3$, for $R > r \ge 1$, we denote $N_{R}$ to be the number of $z\in\zz^d$ lying in $B_R$, and $N_{r,R}$ to be the number of $z\in\zz^d$ lying in $B_R\setminus B_r^o$. Note that
\[
N_{r,R} \le N_R \qx{and} |B_{R-\sqrt{d}/2}|\le N_{R} \le  |B_{R+\sqrt{d}/2}|\le N(d)R^d , \stepcounter{equation} \tag{\theequation}\label{co1}
\]
and for $k=0,1,2,3,\ldots$, we have
\[
N_{2^k,2^{k+1}} \ge N(d) 2^{kd}.\stepcounter{equation} \tag{\theequation}\label{co3}
\]
Then, by H\"older's inequality {for $l_{p_0}$}, \eqref{co1}, and \eqref{co3}, we have
 \begin{align*}
&I_1 - \big( |L^1 u|^{p_0} \big)_{(t_0 -1,t_0)\times C_1(0)}^{1/{p_0}} \\
&\le N \sum_{k=0}^{\infty} \sum_{\substack{|z| = 2^k \\ z\in \zz^d}}^{ 2^{k+1}} 2^{-kd-k\sigma} \big( |L^1 u|^{p_0} \big)_{(t_0 -1,t_0)\times C_1(z)}^{1/{p_0}}\\
&\le N \sum_{k=0}^{\infty} 2^{-kd-k\sigma} N_{2^k,2^{k+1}}  \sum_{\substack{|z| = 2^k \\ z\in \zz^d}}^{ 2^{k+1}} N_{2^k,2^{k+1}}^{-1} \big( |L^1 u|^{p_0} \big)_{(t_0 -1,t_0)\times C_1(z)}^{1/{p_0}}\\
&\le N \sum_{k=0}^{\infty} 2^{-kd-k\sigma} N_{2^k,2^{k+1}} \bigg(  \sum_{\substack{|z| = 2^k \\ z\in \zz^d}}^{ 2^{k+1}} N_{2^k,2^{k+1}}^{-1} \big( |L^1 u|^{p_0} \big)_{(t_0 -1,t_0)\times C_1(z)}\bigg)^{1/{p_0}}\\
&\le N \sum_{k=0}^{\infty} 2^{-k\sigma}  \big( |L^1 u|^{p_0} \big)_{(t_0 -1,t_0)\times B_{2^{k+1}+\sqrt{d}{/2}}}^{1/{p_0}}.
\end{align*}
Thus,
\[
I_1 \le  N \sum_{k=0}^{\infty} 2^{-k\sigma}  \big( |L^1 u|^{p_0} \big)_{(t_0 -1,t_0)\times B_{2^{k}}}^{1/{p_0}}.\stepcounter{equation} \tag{\theequation}\label{co4}
\]
Similarly, by replacing $L^1 u$ with $\lambda u$, we conclude that
\[
I_2 \le  N \sum_{k=0}^{\infty} 2^{-k\sigma}  \big( |\lambda u|^{p_0} \big)_{(t_0 -1,t_0)\times B_{2^{k}}}^{1/{p_0}}.\stepcounter{equation} \tag{\theequation}\label{co444}
\]
For $I_3$, by Proposition \ref{nonhomo}, we have
\begin{align*}
I_3& \le  N \sum_{z \in \zz^d}  (1 + |z|^{d+\sigma})^{-1} \sum_{j = 0}^{\infty} 2^{-j\sigma} \big( |f|^{p_0} \big)_{(t_0 - 1,t_0)\times B_{2^{j}\sqrt{d}}(z)}^{1/{p_0}}\\
&\le  N \sum_{z \in \zz^d}  (1 + |z|^{d+\sigma})^{-1} \sum_{j = 0}^{\infty} 2^{-j\sigma} \big( |f|^{p_0} \big)_{(t_0 - 1,t_0)\times C_{2^{j}}(z)}^{1/{p_0}}.\stepcounter{equation} \tag{\theequation}\label{co2}
\end{align*}
For $j\ge 0$, $z=(z^1,\ldots,z^d) , i=(i^1,\ldots,i^d)\in \zz^d$ and each component of $i$ takes value in $\{0,1,\ldots 2^j -1\}$, we denote
\[
z_{j,i} = (z^1-(2^j-1)/2+i^1,\ldots, z^d-(2^j-1)/2+i^d).
\]
Since $C_1(z_{j,i})$ is disjoint with respect to $i$ for fixed $z$ and $j$, and it is also disjoint with respect to $z$ for fixed $i$ and $j$, we have
\begin{align*}
\sum_{\substack{|z| = 2^k \\ z\in \zz^d}}^{ 2^{k+1}} \int_{(t_0 - 1,t_0)}\int_{C_{2^{j}}(z)}  |f|^{p_0} &= \sum_{\substack{|z| = 2^k \\ z\in \zz^d}}^{ 2^{k+1}} \sum_{i} \int_{(t_0 - 1,t_0)}\int_{C_{1}(z_{j,i})}  |f|^{p_0}\\&\le 2^{\min\{k,j\}d}  \int_{(t_0 - 1,t_0)}\int_{B_{2^{j{-1}}\sqrt{d}+2^{k+1}}}  |f|^{p_0}.\stepcounter{equation} \tag{\theequation}\label{co9}
\end{align*}
By \eqref{co2}, \eqref{co3}, {H\"older's inequality,} and \eqref{co9},
\begin{align*}
&I_3 - {N}\sum_{j = 0}^{\infty} 2^{-j\sigma} \big( |f|^{p_0} \big)_{(t_0 - 1,t_0)\times C_{2^{j}}(0)}^{1/{p_0}}\\
&\le N \sum_{k=0}^{\infty} \sum_{\substack{|z| = 2^k \\ z\in \zz^d}}^{ 2^{k+1}} 2^{-kd-k\sigma} \sum_{j = 0}^{\infty} 2^{-j\sigma} \big( |f|^{p_0} \big)_{(t_0 - 1,t_0)\times C_{2^{j}}(z)}^{1/{p_0}}\\
&\le  N \sum_{k=0}^{\infty} \sum_{j = 0}^{\infty} 2^{-k\sigma-j\sigma} \sum_{\substack{|z| = 2^k \\ z\in \zz^d}}^{ 2^{k+1}} N_{2^k,2^{k+1}}^{-1} \big( |f|^{p_0} \big)_{(t_0 - 1,t_0)\times C_{2^{j}}(z)}^{1/{p_0}}\\
&\le  N \sum_{k=0}^{\infty} \sum_{j = 0}^{\infty} 2^{-k\sigma-j\sigma}  \bigg( 2^{-kd-jd} \sum_{\substack{|z| = 2^k \\ z\in \zz^d}}^{ 2^{k+1}} \int_{(t_0 - 1,t_0)}\int_{C_{2^{j}}(z)}  |f|^{p_0} \bigg)^{1/{p_0}}\\
&\le  N \sum_{k=0}^{\infty} \sum_{j = 0}^{\infty} 2^{-k\sigma-j\sigma}  \bigg( 2^{-kd-jd} 2^{\min\{k,j\}d}  \int_{(t_0 - 1,t_0)}\int_{B_{2^{j{-1}}\sqrt{d}+2^{k+1}}}  |f|^{p_0} \bigg)^{1/{p_0}}\\
&\le  N \sum_{k=0}^{\infty} \sum_{j = 0}^{\infty} 2^{-k\sigma-j\sigma} \bigg(1_{k\ge j}\big( |f|^{p_0} \big)_{(t_0 - 1,t_0)\times B_{2^{k}}}^{1/{p_0}} + 1_{k<j}\big( |f|^{p_0} \big)_{(t_0 - 1,t_0)\times B_{2^{j}}}^{1/{p_0}}\bigg)\\
& \le N \sum_{k=0}^{\infty} 2^{-k\sigma}\big( |f|^{p_0} \big)_{(t_0 - 1,t_0)\times B_{2^k}}^{1/{p_0}} ,
\end{align*}
which implies that
\[
I_3 \le  N \sum_{k=0}^{\infty} 2^{-k\sigma}\big( |f|^{p_0} \big)_{(t_0 - 1,t_0)\times B_{2^k}}^{1/{p_0}}. \stepcounter{equation} \tag{\theequation}\label{co5}
\]
Thus, combining \eqref{e1}, \eqref{i2}, \eqref{co4}, \eqref{co444}, and \eqref{co5}, we arrive at \eqref{i3}, and the proposition is proved.
\end{proof}

Next, we introduce the dyadic cubes as follows:
for each $n \in \zz$, pick $k(n){\in \zz}$ such that
$$ k(n) \le \sigma n < k(n) +1,$$ and let
\begin{align*}
     Q_{\vec{i}}^n = \big[ \frac{i_0}{2^{k(n)}} + T, \frac{i_0+1}{2^{k(n)}} + T \big) \times \big[ \frac{i_1}{2^n}, \frac{i_1+1}{2^n} \big) \times \cdots \times \big[ \frac{i_d}{2^n}, \frac{i_d+1}{2^n} \big)
\end{align*} and
$$
\cc_n := \big\{ Q_{\vec{i}}^n = Q_{(i_0,\ldots,i_d)}^n : \vec{i} = (i_0,\ldots,i_d)\in \zz^{d+1}, i_0\le -1\big\}.
$$
Furthermore, denote the dyadic sharp function of $g$ by
$$  g^{\#}_{dy}(t,x) = \sup_{n< \infty} \fint_{Q_{\vec{i}}^n \ni (t,x) } |g(s,y) - g_{|n}(t,x)| \, dy \, ds,$$
where
$$g_{|n}(t,x) = \fint_{Q_{\vec{i}}^n} g(s,y) \, dy \, ds \quad \text{for}\quad  (t,x) \in Q_{\vec{i}}^n.
$$

\begin{proof}[Proof of Theorem \ref{main2}]
{\em Step 1. The a priori estimate.}

We first prove \eqref{lo1} under the assumption that $u$ is compactly supported in the spacial variables. For the given $w_1 \in A_p(\rr, dt)$ and $w_2 \in A_q (\rr^d, dx)$, using reverse H\"older's inequality for $A_p$ weights, we pick
$
\gamma_1 = \gamma_1 (d, p, M_1)$ and $
\gamma_2 = \gamma_2 (d, q, M_1)
$
such that $p - \gamma_1 > 1$, $q - \gamma_2 > 1$, and
\[
w_1 \in A_{p- \gamma_1} (\rr, dt),
\quad w_2 \in A_{q - \gamma_2} (\rr^d, dx).
\]
Take $p_0 = p_0 (d, p, q, M_1)  \in (1, \infty)$ so that
\[
p_0 = \min \Bigg\{
\frac{p}{p- \gamma_1}, \frac{q}{q - \gamma_2}
\Bigg\} > 1.
\]
Note that
\[
w_1 \in A_{p - \gamma_1}(\rr, dt)
\subset A_{\frac{p}{p_0}} (\rr, dt)
\]
and
\[
w_2 \in A_{q - \gamma_2}(\rr^d, dx)
\subset A_{\frac{q}{p_0}} (\rr^d, dx).
\]
From these inclusions and the fact that $u \in \hh_{p,q,w,0}^{1, \sigma} ( \rr^d_T )$, it follows that
$u \in \hh_{p_0,0,\mathrm{loc}}^{1, \sigma} ( \rr^d_T )$. See the proof of \cite[Lemma 5.10]{Ap} for details.
Furthermore, since $u$ is compactly supported, we have $u \in \hh_{p_0,0}^{1, \sigma} ( \rr^d_T )$, and $u \in \hh_{p_0}^{1, \sigma} ( (-\infty, 0) \times \rr^d )$ by taking the zero extension for $t<0$.
Therefore, by Proposition \ref{4.5}, for any $  r\in (0, \infty)$, $\kappa \in (0, 1/4)$, and $(t_1, x_0) \in (0,T] \times \rr^d$, we have
\begin{align*}
& ( |L^1 u - (L^1 u )_{{Q_{\kappa r}(t_1,x_0)}} | )_{{Q_{\kappa r}(t_1,x_0)}
} + ( |\lambda u - (\lambda u )_{{Q_{\kappa r}(t_1,x_0)}} | )_{{Q_{\kappa r}(t_1,x_0)}
}\\
&\quad \leq N \kappa^\tau \sum_{k=0}^{\infty} 2^{-k\sigma} \big[ \big( |L^1 u|^{p_0} \big)_{(t_1 -r^\sigma,t_1)\times B_{2^{k}r} (x_0)}^{1/{p_0}} + \big( |\lambda u|^{p_0} \big)_{(t_1 -r^\sigma,t_1)\times B_{2^{k}r} (x_0)}^{1/{p_0}} \big]\\
&\quad + N \kappa^{- (d+\sigma )/p_0}\sum_{k = 0}^{\infty}
2^{-k\sigma} \big( |f|^{p_0} \big)_{(t_1 - r^\sigma,t_1)\times B_{2^kr}(x_0)}^{1/{p_0}}, \stepcounter{equation}\tag{\theequation}\label{io3}
\end{align*}
where $\tau = \tau (d,\sigma, p,q,M_1)$ and $N = N(d,\nu,\Lambda,\sigma,p,q,M_1)$.

Note that for any $(t_0,x_0)\in \rr^d_T$, $n\in \zz$, and $ Q_{\vec{i}}^n$ containing $(t_0,x_0)$, there exist $r = r(d,\sigma, n){>0}$ and $t_1 = \min(T,t_0+r^\sigma/2)\in(-\infty,T]$ such that
$$
Q_{\vec{i}}^n \subset Q_r(t_1, x_0)\quad \text{and}\quad
|Q_r(t_1, x_0)|\le N(d,\sigma)|Q_{\vec{i}}^n|.
$$
Thus, by (\ref{io3}), we have
\begin{align*}
&\fint_{Q_{\vec{i}}^n \ni (t_0,x_0) } |L^1 u(s,y) - {(L^1 u)}_{|n}(t_0,x_0)|+ |\lambda u(s,y) - {(\lambda u)}_{|n}(t_0,x_0)| \, dy \, ds\\
&\leq N \kappa^\tau \sum_{k=0}^{\infty} 2^{-k\sigma} \big[  \big( |L^1 u|^{p_0} \big)_{(t_1 -(r/\kappa)^\sigma,t_1)\times B_{2^kr/\kappa} (x_0)}^{1/{p_0}}+   \big( |\lambda u|^{p_0} \big)_{(t_1 -(r/\kappa)^\sigma,t_1)\times B_{2^kr/\kappa} (x_0)}^{1/{p_0}}\big]\\
&\quad + N \kappa^{- (d+\sigma )/p_0}\sum_{k = 0}^{\infty}
2^{-k\sigma} \big( |f|^{p_0} \big)_{(t_1 - r^\sigma,t_1)\times B_{2^kr/\kappa}(x_0)}^{1/{p_0}}\\
&\leq N \bigg(\kappa^\tau (\css \cmm | L^1 u |^{p_0} )^{\frac 1 {p_0}} + \kappa^\tau (\css \cmm | \lambda u |^{p_0} )^{\frac 1 {p_0}}  + \kappa^{- (d+\sigma )/p_0}
(\css \cmm | f |^{p_0} )^{\frac 1 {p_0}} \bigg)(t_0, x_0),
\end{align*}
where $N$ is independent of $n$ {and $\kappa$}, and for the last inequality, we used that for any function $g$ and $A \subset B$,
\[
\Bigg|  \fint_A g - \fint_B g \Bigg|
\leq \frac{|B|}{|A|} \fint_B | g - (g)_B|.
\]
Therefore,
\begin{align*}
&(L^1 u)^{\#}_{dy}(t_0,x_0) + (\lambda u)^{\#}_{dy}(t_0,x_0)\\
&\leq N \bigg(\kappa^\tau (\css \cmm | L^1 u |^{p_0} )^{\frac 1 {p_0}} + \kappa^\tau (\css \cmm | \lambda u |^{p_0} )^{\frac 1 {p_0}}  + \kappa^{- (d+\sigma )/p_0}
(\css \cmm | f |^{p_0} )^{\frac 1 {p_0}} \bigg)(t_0, x_0).
\end{align*}
Then, by the weighted sharp function theorem \cite[Corollary 2.7]{Ap} and the weighted maximal function theorem for strong maximal functions \cite[Theorem 5.2]{dong20},
\begin{align*}
\norm{L^1 u}_{p,q,w} + \norm{\lambda u}_{p,q,w} &\le N \kappa^\tau \norm{L^1 u}_{p,q,w} + N \kappa^\tau \norm{\lambda u}_{p,q,w} + N \kappa^{- (d+\sigma )/p_0} \norm{f}_{p,q,w},
\end{align*}
where $N = N (d,\nu,\Lambda,\sigma,p,q,M_1)$. Therefore, by taking a sufficiently small $\kappa < 1/4$ such that
$$
N \kappa^{\tau} <1/2,
$$
we get
\[\norm{L^1 u}_{L_{p,q,w} } + \lambda\norm{ u}_{L_{p,q,w} }
\leq
N \norm{f}_{L_{p,q,w} }. \]
Since the space of functions that are compactly supported in the spacial variables is dense in $\hh_{p,q,w,0}^{1,\sigma} ( T )$, we obtain \eqref{lo1} and the continuity of $\partial_t - L  $. Then \eqref{uo1} follows from the Equation \eqref{eqn6} together with the triangle in equality.

{\em Step 2. Existence of solutions.}

By the method of continuity, it suffices to prove the existence of solutions for the simple equation with $L = -\pf$ and $\lambda = 0$. By the a priori estimate proved in Step 1 and the density of smooth functions in Lebesgue spaces, without loss of generality, we assume that $f\in C_0^\infty((0,\infty)\times \rr^d)$.
Let \[u (t,x)= \int_0^t\int_{\rr^d} \zeta(t-s,x-y)f(s,y)\,dy\,ds, \stepcounter{equation}\tag{\theequation}\label{ur1}\]
where $\hat\zeta(t,\xi)= e^{-t|\xi|^\sigma}$.
Then by \cite[Lemma 4.1]{kim20},
\[
\partial_t u + \pf u= f.
\]

When $p=q$, it follows from \cite[Theorem 5.14]{LP}, a Fourier multiplier theorem, that
\[
\norm{\partial_t u}_{L_{p,w}(\rr^d_T) }+
\norm{\pf u}_{L_{p,w}(\rr^d_T) }
\leq
N \norm{f}_{L_{p,w}(\rr^d_T) }. \]
For general $p$ and $q$, using the same definition of $u$ as in \eqref{ur1} together with the extrapolation theorem in \cite[Theorem 2.5]{Ap} or \cite[Theorem 1.4]{extrapolation}, we have
\[
\norm{\partial_t u}_{L_{p,q,w} (\rr^d_T)}+
\norm{\pf u}_{L_{p,q,w} (\rr^d_T)}
\leq
N \norm{f}_{L_{p,q,w}(\rr^d_T) }, \]
which implies that $u \in H^{1,\sigma}_{p,q,w,0}(\rr^d_T)$. The theorem is proved.
\end{proof}

Next, we prove Corollary \ref{newco1} using the frozen coefficient argument and a partition of unity.
\begin{corollary}\label{9op}
Let $\beta \in (0,1)$, $\alpha \in (0,1]$, $\sigma \in (0,2)$, $T \in (0, \infty)$, $p \in (1, \infty)$, $\lambda \ge 0$, $M_1 \in [1, \infty)$, and $[w]_{p,p} \le M_1$.
There exists $r_0 = r_0(d, \nu,\Lambda, \alpha,\sigma, p , \beta, \omega,M_1) >0$ such that under Assumptions \ref{aaa1} and \ref{aaa2}, for any $u \in \hh_{p,w,0}^{\alpha, \sigma} ( \rr^d_T )$ supported on $(0,T) \times B_{r_0 }(x_1)$ for some $x_1\in \rr^d$ and satisfies
\[
\partial_t u - L u + \lambda u = f \qx{in } \rr^d_T,
\] we have
\begin{align*}
&\norm{\partial_t u}_{L_{p,w}  ( \rr^d_T )}
+ \norm{\pf u}_{L_{p,w}  ( \rr^d_T )}  + \lambda  \norm{u}_{L_{p,w}  ( \rr^d_T )}\\
&\leq
N \norm{f}_{L_{p,w}  ( \rr^d_T ) }+ N \norm{u}_{L_{p,w}  ( \rr^d_T ) },
\end{align*}
where $N = N(d, \nu, \Lambda, \sigma, p,M_1,\beta,\omega)$.
\end{corollary}
\begin{proof}
This follows from {\cite[Lemma 5.1]{dyk}} and Theorem \ref{main2} with the frozen coefficient argument. {See the proof of \cite[Lemma 5.2]{dyk} for details.}
\end{proof}

\begin{proof}[Proof of Corollary \ref{newco1}]
First, the continuity of $\partial_t - L $ follows by \cite[Corollary 2.5]{dyk}. By Theorem \ref{main} together with the method of continuity, it remains to prove the a priori estimate \eqref{up6}.

We first consider the case when $b=c= 0$ and $p=q$ with general weight $w(t,x) = w_1(t)w_2(x)$. We take $r_0$ from Corollary \ref{9op}, and use a partition of unity argument with respect to the spatial variables. Let $\zeta \in C_0^\infty (B_{r_0})$ satisfy
$$\norm{\zeta}_{L_p} = 1, \quad \norm{D_x \zeta}_{L_p} \le N ,\quad \norm{D^2_x\zeta}_{L_p}  \leq N, \qx{and} \zeta_z(x) = \zeta(x-z),$$
where $z\in \rr^d$ and $N =N (d, \nu, \Lambda, \sigma, p,\beta ,\omega)$.
It follows that
\begin{align*}
\partial_t(\zeta_z u) -L (\zeta_z u) + \lambda \zeta_z u &= \zeta_zf +\zeta_z Lu -L(\zeta_z u).
\end{align*}
By Corollary \ref{9op}, for any $\ep\in(0,1)$,
\begin{align*}
&\norm{\partial_t u }_{L_{p,w}(\rr^d_T)}^p  + \norm{\pf u }_{L_{p,w}(\rr^d_T)}^p +  \norm{\lambda u }_{L_{p,w}(\rr^d_T)}^p  \\
&= \int_0^T\int_{\rr^d} \int_{\rr^d} \big(|\zeta_z(x)\partial_t u(t,x)|^p + |\zeta_z(x)\lambda u(t,x)|^p\big) w(t,x)  w(t,x)\,dz \,dx  \,dt\\
&\quad + \int_0^T\int_{\rr^d} \int_{\rr^d} |\zeta_z(x)\pf u(t,x)|^p w(t,x)\,dz \,dx  \,dt\\
&\le \int_0^T\int_{\rr^d} \int_{\rr^d} \big( |\partial_t \big( \zeta_z(x) u(t,x)\big)|^p + |\lambda \zeta_z(x) u(t,x)|^p \big) w(t,x)\,dz \,dx  \,dt\\
&\quad + N \int_0^T\int_{\rr^d} \int_{\rr^d} |\pf(\zeta_z u)(t,x)|^p w(t,x) \,dx \,dz \,dt\\
&\quad + N \int_0^T\int_{\rr^d} \int_{\rr^d} |\zeta_z(x)\pf u(t,x) -\pf(\zeta_z u)(t,x)|^pw(t,x) \,dx \,dz\,dt\\
&\le N\int_0^T \int_{\rr^d} \int_{\rr^d} \big(|\zeta_z(x) f(t,x)|^p + |\zeta_z(x) u(t,x)|^p \big) w(t,x)\,dx \,dz\,dt
\\
&\quad +  N \int_0^T\int_{\rr^d} \int_{\rr^d} |\zeta_z(x)L u(t,x) -L(\zeta_z u)(t,x)|^p w(t,x)\,dx \,dz\,dt\\
&\quad + N \int_0^T\int_{\rr^d} \int_{\rr^d} |\zeta_z(x)\pf u(t,x) -\pf(\zeta_z u)(t,x)|^pw(t,x) \,dx \,dz\,dt\\
&\le N\norm{f}^p_{L_{p,w}(\rr^d_T)} + N \norm{u}^p_{L_{p,w}(\rr^d_T)}
+ 1_{\sigma>1} N \norm{Du}^p_{L_{p,w}(\rr^d_T)} + 1_{\sigma=1} \ep  \norm{Du}^p_{L_{p,w}(\rr^d_T)}.
\stepcounter{equation}\tag{\theequation}\label{0pi}\end{align*}
Indeed, for the last inequality, to estimate
\begin{align*}
\int_0^T\int_{\rr^d} \int_{\rr^d} |\zeta_z(x)L u(t,x) -L(\zeta_z u)(t,x)|^p w_2(x) w_1(t)\,dx \,dz\,dt,
\end{align*}
we consider the following three cases depending on the value of $\sigma$.
When $\sigma \in(0,1)$,
\begin{align*}
&|\zeta_z(x)L u(t,x) -L(\zeta_z u)(t,x)| \\
&\le N \bigg(\int_{B_1} + \int_{B^c_1} \bigg) |\zeta_z(x+y) - \zeta_z(x)\|u(t,x+y) \|y|^{-d-\sigma}\,dy=: I^1_{1} +  I^1_{2}.
\end{align*}
Since
\begin{align*}
I^1_{1}
&\le \int_{B_1} |u(t,x+y) \|y|^{-d-\sigma{+1}}   \int_0^1|D\zeta_z(x+sy)| \,ds \,dy,
\end{align*}
by the Minkowski inequality, \cite[Lemma 3.2]{dyk}, {and the weighted Hardy-Littlewood maximal function theorem,} we conclude that
\begin{align*}
&\int_0^T\int_{\rr^d} \int_{\rr^d}(I^1_{1}(t,x,z ))^p w_2(x) w_1(t)\,dx \,dz\,dt, \\
&\le \norm{ D\zeta}^p_{L_p(\rr^d)} \int_0^T\int_{\rr^d} \big(\int_{B_1} |u(t,x+y)| |y|^{-d-\sigma + 1}\,dy\big)^p w_2(x)w_1(t) \,dx \,dt,\\
&\le N \int_0^T\int_{\rr^d}\cmm_x u(t,x)^p w_2(x) w_1(t) \,dx \,dt \le  N \norm{u}^p_{L_{p,w}(\rr^d_T)},
\end{align*}
and
\begin{align*}
&\int_0^T\int_{\rr^d} \int_{\rr^d}(I^1_{2}(t,x,z ))^p w_2(x) w_1(t)\,dx \,dz\,dt, \\
&\le N \norm{ \zeta}^p_{L_p(\rr^d)} \int_0^T\int_{\rr^d} \big(\int_{B^c_1} |u(t,x+y)| |y|^{-d-\sigma }\,dy\big)^p w_2(x)w_1(t) \,dy \,dt \\
&{\le N \norm{ \zeta}^p_{L_p(\rr^d)} \int_0^T\int_{\rr^d} \cmm_x u(t,x)^p w_2(x)w_1(t) \,dy \,dt} \\
&\le  N \norm{u}^p_{L_{p,w}(\rr^d_T)}.
\end{align*}
{Here $\cmm_x u(t,\cdot)$ stands for the maximal function of $u$ with respect to the $x$ variable.}
Furthermore, using a similar decomposition {as} in the proof of Lemma \ref{com} together with the maximal function as above, when $\sigma \in(1,2)$,
\begin{align*}
I \le N \norm{u}^p_{L_{p,w}(\rr^d_T)}
+ 1_{\sigma>1} N \norm{Du}^p_{L_{p,w}(\rr^d_T)},
\end{align*}
and when $\sigma = 1 $,
 \begin{align*}
I \le N(\ep) \norm{u}^p_{L_{p,w}(\rr^d_T)} + \ep  \norm{Du}^p_{L_{p,w}(\rr^d_T)}.
\end{align*}
Therefore, by \eqref{0pi} together with Lemma \ref{a1} or picking {a} sufficiently small $\ep$ when $\sigma =1$, we obtain
\begin{align*}
&\norm{\partial_t u}_{L_{p,w}(\rr^d_T)} + \norm{\pf u}_{L_{p,w}(\rr^d_T)}+\lambda \norm{ u}_{L_{p,w}(\rr^d_T)}  \\
&\leq
N\norm{f}_{L_{p,w}(\rr^d_T)} + N \norm{u}_{L_{p,w}(\rr^d_T)}. \stepcounter{equation}\tag{\theequation}\label{uy2}
\end{align*}
By taking $\lambda_0 = \max(2N,1)$, for all $\lambda \ge \lambda_0$, we have \eqref{89oii}.
On the other hand, when $\lambda= 0$, by subdividing $(0,T)$ into sufficiently small sub-intervals, taking cutoff functions in time, and applying an induction argument as in \cite[Theorem 2.2]{dong20}, we get
\begin{align*}
\norm{ u}_{L_{p,w}(\rr^d_T)}  \leq
N(T)\norm{f}_{L_{p,w}(\rr^d_T)},
\end{align*}
which together with \eqref{uy2} implies \eqref{89oi}.

For general $b$ and $c$ and $p=q$, we have
\[
\partial_t u
- L u
= f -b^iD_iu 1_{\sigma \ge 1}- cu.
\]
Thus, if $\sigma > 1$ {\eqref{uy2}} follows from the case when $b=c = 0$ and Lemma \ref{a1}.
If $\sigma = 1$, then
\begin{align*}
&\norm{\partial_t u}_{L_{p,w}(\rr^d_T)} + \norm{Du}_{L_{p,w}(\rr^d_T)}\\
&\leq
N\norm{f}_{L_{p,w}(\rr^d_T)} + N \norm{u}_{L_{p,w}(\rr^d_T)} + N \norm{bDu}_{L_{p,w}(\rr^d_T)}.\stepcounter{equation} \tag{\theequation}\label{uy}
\end{align*}
If $b$ is sufficiently small such that $N\norm{b}_{L_\infty(\rr^d_T)} \le 1/2$, then by absorbing the last term on the right-hand side of \eqref{uy} to the left, we obtain \eqref{uy2}. On the other hand, if $\alpha =1$ and $b$ is uniformly continuous, we take
\[
B(t) = \int_0^t b(s,0) \, ds, \quad \Tilde{b}(t,x) = b(t,x-B(t)), \quad \Tilde{c}(t,x) = c(t,x-B(t)),\]
and \[
\quad \Tilde{u}(t,x) = u(t,x-B(t)), \quad \Tilde{f}(t,x) = f(t,x-B(t)).
\]
It follows that
\[
\partial_t \Tilde{u}(t,x) - \Tilde{L}\Tilde{u}(t,x) + (\Tilde{b}^i(t,x) - b^i(t,0))D_i\Tilde{u}(t,x) + \Tilde{c} \Tilde{u}(t,x) = \Tilde{f}(t,x),\stepcounter{equation}\tag{\theequation}\label{uy3}
\]
where $\Tilde{L}$ is the operator with the kernel $\Tilde{K}(t,x,y) = K(t,x-B(t),y)$. Then, by moving the last two terms on the left-hand side of \eqref{uy3} to the right-hand side and applying the estimate to $\Tilde{u}$ together with a change of variables, we have
\begin{align*}
\norm{D u}_{L_{p,w}(\rr^d_T)} \leq
N\norm{f}_{L_{p,w}(\rr^d_T)} + N \norm{u}_{L_{p,w}(\rr^d_T)} + N \norm{(b - b(\cdot,0))Du}_{L_{p,w}(\rr^d_T)},
\end{align*}
which implies that there exists $r_0>0$ depending on the modulus of continuity of $b$ such that if $u$ is supported on $(0,T) \times B_{r_0 }$, then
\begin{align*}
 \norm{Du}_{L_{p,w}(\rr^d_T)} \leq
N\norm{f}_{L_{p,w}(\rr^d_T)} + N \norm{u}_{L_{p,w}(\rr^d_T)}.
\end{align*}
Therefore, by a partition of unity and \eqref{uy}, we arrive at \eqref{uy2}, and \eqref{89oi} follows.

Finally, for general $p,q$, we use the extrapolation theorems in \cite[Theorem 2.5]{Ap} or \cite[Theorem 1.4]{extrapolation}. The Corollary is proved.
\end{proof}

\begin{proof}[Proof of Corollary \ref{newco}]
With Theorem \ref{main}, the proof of Corollary \ref{newco} proceeds similar to that of Corollary \ref{newco1}, using the frozen coefficient argument and a partition of unity.
\end{proof}

\appendix
\section{}\label{A}
\begin{lemma}[An interpolation inequality in the spatial variables]\label{a1}
Let $\sigma \in (1,2)$, $p\in (1,\infty)$, and $u \in H_p^{\sigma}(\rr^d)$. For any $\ep>0$,
\begin{align*}
\norm{Du}_{L_p(\rr^d)} \le N (d,\sigma,p) \ep^{1/(1- \sigma)}\norm{u}_{L_p(\rr^d)} + \ep\norm{(-\Delta)^{\sigma/2}u}_{L_p(\rr^d)}.\stepcounter{equation}\tag{\theequation}\label{int1}
\end{align*}
\end{lemma}
\begin{proof}
By taking $f : = u + (-\Delta)^{\sigma/2}u$, we have
$$
\cff(Du)(\xi) = \frac{i\xi}{1 + |\xi|^{\sigma}}\cff(f)(\xi) =: m(\xi)\cff(f)(\xi).
$$
Since $\sigma >1$, the multiplier $m(\xi)$ satisfies $$\sup_{\xi\in\rr^d} |\xi|^{|\gamma|}|\partial^{\gamma}m| \le N(d,\sigma,p),$$ for any multi-index $\gamma$ of length $|\gamma|\le d/2+1$. Thus, by the Mikhlin multiplier theorem, we have
\begin{align*}
    \norm{Du}_{L_p(\rr^d)} \le N\norm{f}_{L_p(\rr^d)} \le N\norm{u}_{L_p(\rr^d)} + N\norm{(-\Delta)^{\sigma/2}u}_{L_p(\rr^d)},\stepcounter{equation}\tag{\theequation}\label{int}
\end{align*}
where $N = N (d,\sigma,p)$. Then taking $v(\cdot) := u(\ep\cdot)$ and applying \eqref{int} to $v$ yield
\begin{align*}
    \ep\norm{Du}_{L_p(\rr^d)} \le N\norm{u}_{L_p(\rr^d)} + N\ep^{\sigma} \norm{(-\Delta)^{\sigma/2}u}_{L_p(\rr^d)},
\end{align*}
which implies \eqref{int1}. The lemma is proved.
\end{proof}
\begin{lemma}\label{a2}
Let $\alpha \in (0,1]$, $T \in (0, \infty)$, $p,q \in (1, \infty)$.
For any $u \in \hh_{p,q,0}^{\alpha, \sigma} ( T )$, we have
$$\norm{u}_{L_{p,q} ( \rr^d_T )} \le N T^\alpha \norm{\partial_t^\alpha u}_{L_{p,q} ( \rr^d_T )},  $$
where $N = N(\alpha,p,q)$.
\end{lemma}

\begin{proof} Note that
$u = I^\alpha \partial_t^\alpha u.$
Therefore, the result follows from \cite[Lemma 5.5]{dong20}.
\end{proof}

\begin{lemma}\label{com}
Let $L$ satisfy \eqref{sig1} when $\sigma=1$ and \eqref{ass}. Recall the definition of $\tilde\sigma$ in \eqref{bdds}. For $k = 0,1,\ldots,$, let $\zeta_k$ be the cutoff functions defined in \eqref{cutt}. Under the same assumptions and the same notation as in Lemma \ref{local}, we have the following estimates for $I_k := \norm{\zeta_k Lu -L(\zeta_k u) }_{L_p(\rr^d_T)}$.
\begin{enumerate}
    \item When $\sigma\in(0,1)$,
\begin{align*}
I_k\le   N  \frac{2^{\sigma k}}{(R-r)^{\sigma}} \norm{u}_{p,R}  + N \frac{2^{(d+\sigma)k}}{(R-r)^{d+\sigma}}R^{d/p}(1 + R^{d+\sigma} ) \norm{u}_{L_p((0,T);L_1(\rr^d,\psi))}; \stepcounter{equation}\tag{\theequation}\label{a41}
\end{align*}
    \item when $\sigma\in(1,2)$,
    \begin{align*}
I_k &\le  N \frac{2^{(\sigma-1) k}}{(R-r)^{\sigma-1}} \norm{Du}_{p,r_{k+3}} + N  \frac{2^{\sigma k}}{(R-r)^{\sigma}} \norm{u}_{p,R}  \\
&\quad+ N \frac{2^{(d+\sigma)k}}{(R-r)^{d+\sigma}}R^{d/p}(1 + R^{d+\sigma} ) \norm{u}_{L_p((0,T);L_1(\rr^d,\psi))}; \stepcounter{equation}\tag{\theequation}\label{a42}
\end{align*}
    \item when $\sigma =1 $, for any $\ep \in (0,1)$,
    \begin{align*}
I_k   &\le \ep^3\norm{Du}_{p,r_{k+3}} +  N\frac{2^{k}}{R-r}\ep^{-3} \norm{u}_{p,R} \\
&\quad +  N R^{d/p}\bigg(1+ \frac{2^{(d+1)k}\ep^{-3(d+1)}}{(R-r)^{d+1}}  (1 + R^{d+1} )\bigg) \norm{u}_{L_p((0,T);L_1(\rr^d,\psi))}, \stepcounter{equation}\tag{\theequation}\label{a43}
\end{align*}
\end{enumerate}
where $N = N(d,\Lambda,\tilde\sigma)$ and $\norm{\, \cdot\, }_{p,r} := \norm{\,\cdot\,}_{L_p((0,T) \times B_r)}$ for any $r> 0 $.
\end{lemma}
\begin{proof}
In this proof, we use $N$ to denote a constant which may depend on $d$, $\Lambda$, and $\tilde\sigma$. Note that $N$ does not depend on $\nu$, the lower bound of the operator.

Next, let $\Tilde{r}_k = r_{k+3} - r_{k+2}$, and note that
\begin{align*}
&L(\zeta_k u)(t,x) - \zeta_k Lu(t,x) \\
&= \int_{\rr^d}\bigg((\zeta_k(x+y) - \zeta_k(x))u(t,x+y) - y\cdot \nabla\zeta_k(x) u(t,x)\chi^{(\sigma)}(y)\bigg)K(t,y)\,dy. \stepcounter{equation}\tag{\theequation}\label{long}
\end{align*}
We are going to split the proof into 3 cases depending on the value of $\sigma$.

{\em Case 1: $\sigma\in(0,1)$.}
In this case, by \eqref{long},
\begin{align*}
&|L(\zeta_k u)(t,x) - \zeta_k Lu(t,x)|\\
&\le N \int_{\rr^d}|\zeta_k(x+y) - \zeta_k(x)\|u(t,x+y) \|y|^{-d-\sigma}\,dy\\
&\le N \bigg(\int_{B_{\Tilde{r}_k}} + \int_{B^c_{\Tilde{r}_k}} \bigg) |\zeta_k(x+y) - \zeta_k(x)\|u(t,x+y) \|y|^{-d-\sigma}\,dy=: I^1_{k,1} +  I^1_{k,2}.
\end{align*}
For $I^1_{k,1}$, since $y \in B_{\Tilde{r}_k} $, we have
\[|\zeta_k(x+y) - \zeta_k(x)|
\le  \norm{D\zeta_k}_{L_{\infty}}|y|1_{|x|<r_{{k+2}}}
\le N \frac{2^k}{R-r}|y|1_{|x|<r_{{k+2}}}, \stepcounter{equation}\tag{\theequation}\label{001}\]
which together with the Minkowski inequality implies that
\begin{align*}
\norm{I^1_{k,1}}_{L_p(\rr^d_T)} &\le N\frac{2^k}{R-r} \int_{B_{\Tilde{r}_k}} \norm{u(\cdot,\cdot +y)}_{p,r_{k+2}}  |y|^{1-d-\sigma} \, dy \\
&\le N\frac{2^k}{R-r}  \norm{u}_{p,R} \int_{B_{\Tilde{r}_k}} |y|^{1-d-\sigma} \, dy \\
&\le N (d,\sigma_1)\frac{2^k}{R-r}  \Tilde{r}_k^{1-\sigma}\norm{u}_{p,R}\le N \frac{2^{\sigma k}}{(R-r)^{\sigma}} \norm{u}_{p,R}.
\stepcounter{equation}\tag{\theequation}\label{tt1}\end{align*}
On the other hand,
\begin{align*}
    I^1_{k,2} \le \int_{B^c_{\Tilde{r}_k}} \bigg(1_{|x+ y|< r_{{k+1}}}+ 1_{|x|< r_{{k+1}}}\bigg) |u(t,x+y)| |y|^{-d-\sigma}  \, dy:= I^1_{k,2,1} + I^1_{k,2,2},
\end{align*}
\begin{align*}
\norm{I^1_{k,2,1}}_{L_p(\rr^d_T)} \le   \int_{B^c_{\Tilde{r}_k}} |y|^{-d-\sigma} \, dy \norm{u}_{p,R} \le N(d,\sigma_0)\Tilde{r}_k^{-\sigma}\norm{u}_{p,R} \le N  \frac{2^{\sigma k}}{(R-r)^{\sigma}} \norm{u}_{p,R}. \stepcounter{equation}\tag{\theequation}\label{tt2}
\end{align*}
By H\"older's inequality and the Minkowski inequality,
\begin{align*}
\norm{I^1_{k,2,2}}_{L_p(\rr^d_T)}& \le N(d) r^{d/p}_{k+1} \bigg(\int_0^T \norm{I^1_{k,2,2}(t,\cdot)}_{L_\infty(B_{r_{k+1}})}^p \, dt\bigg)^{1/p}\\
&\le N r^{d/p}_{k+1} \frac{1 + r_{k+3}^{d+\sigma}}{\Tilde{r}_k^{d+\sigma}} \bigg(\int_0^T \norm{u(t,\cdot)}_{L_1(\rr^d,\psi)}^p \, dt\bigg)^{1/p}\\
&=  N r^{d/p}_{k+1} \frac{1 + r_{k+3}^{d+\sigma}}{\Tilde{r}_k^{d+\sigma}}  \norm{u}_{L_p((0,T);L_1(\rr^d,\psi))}\\
&\le N \frac{2^{(d+\sigma)k}}{(R-r)^{d+\sigma}}R^{d/p}(1 + R^{d+\sigma} ) \norm{u}_{L_p((0,T);L_1(\rr^d,\psi))},\stepcounter{equation}\tag{\theequation}\label{tt3}
\end{align*}
where for the second inequality, we used the fact that if $|y| \ge \Tilde{r}_k$ and $|x|< r_{k+1}$ then
$$\frac{1 + |x+y|^{d+ \sigma}}{|y|^{d+ \sigma}}\le \frac{1+ (r_{k+1}+ \Tilde{r}_k)^{d+\sigma}}{\Tilde{r}_k^{d+\sigma}}\le \frac{1 + r_{k+3}^{d+\sigma}}{\Tilde{r}_k^{d+\sigma}}.$$
Combining \eqref{tt1}, \eqref{tt2}, and \eqref{tt3}, we conclude
\begin{align*}
I_k &\le \norm{I^1_{k,1}}_{L_p(\rr^d_T)} + \norm{I^1_{k,2}}_{L_p(\rr^d_T)} \\
&\le N  \frac{2^{\sigma k}}{(R-r)^{\sigma}} \norm{u}_{p,R}  + N \frac{2^{(d+\sigma)k}}{(R-r)^{d+\sigma}}R^{d/p}(1 + R^{d+\sigma} ) \norm{u}_{L_p((0,T);L_1(\rr^d,\psi))}.
\end{align*}

{\em Case 2: $\sigma \in (1,2)$.} In this case, by \eqref{long},
\begin{align*}
&|L(\zeta_k u)(t,x) - \zeta_k Lu(t,x)|\\
&\le N(\Lambda)(2-\sigma) \int_{\rr^d}|\zeta_k(x+y) - \zeta_k(x)|
|u(t,x+y) - u(t,x)| |y|^{-d-\sigma}\,dy\\
&\quad +N(\Lambda)(2-\sigma) \int_{\rr^d}|\zeta_k(x+y) - \zeta_k(x)- y\cdot \nabla\zeta_k(x)| |u(t,x)||y|^{-d-\sigma} \,dy\\
&=: I^2_{k,1} +  I^2_{k,2}.
\end{align*}
By \eqref{001} and the fundamental theorem of calculus,
\begin{align*}
I^2_{k,1}
&\le  N(2-\sigma) \bigg(\int_{B_{\Tilde{r}_k}} + \int_{B^c_{\Tilde{r}_k}} \bigg) |\zeta_k(x+y) - \zeta_k(x)||u(t,x+y) - u(t,x)||y|^{-d-\sigma}\,dy\\
&\le N(2-\sigma)  \frac{2^k}{R-r} \int_{B_{\Tilde{r}_k}} \int_0^1 1_{|x|<r_{{k+2}}} |\nabla u(t,x + sy)| |y|^{2-d-\sigma}\,ds \,dy\\
&\quad + N \int_{B^c_{\Tilde{r}_k}} \bigg(1_{|x+ y|< r_{{k+1}}}+ 1_{|x|< r_{{k+1}}}\bigg) |u(t,x+y) - u(t,x)||y|^{-d-\sigma}  \, dy.
\end{align*}
Similar to \eqref{tt1}, \eqref{tt2}, and \eqref{tt3}, we have
\begin{align*}
\norm{I^2_{k,1}}_{L_p(\rr^d_T)} &\le  N \frac{2^{(\sigma-1) k}}{(R-r)^{\sigma-1}} \norm{Du}_{p,r_{k+3}} + N  \frac{2^{\sigma k}}{(R-r)^{\sigma}} \norm{u}_{p,R} \\
&\quad + N \frac{2^{(d+\sigma)k}}{(R-r)^{d+\sigma}}R^{d/p}(1 + R^{d+\sigma} ) \norm{u}_{L_p((0,T);L_1(\rr^d,\psi))}.\stepcounter{equation}\tag{\theequation}\label{ee1}
\end{align*}
For $I^2_{k,2}$, when $y \in B_{\Tilde{r}_k} $, the mean value theorem leads to
\begin{align*}
&|\zeta_k(x+y) - \zeta_k(x)- y\cdot \nabla\zeta_k(x)| \\
&\le  \norm{D^2 \zeta_k}_{L_{\infty}}|y|^21_{|x|<r_{{k+2}}}
\le N \frac{2^{2k}}{(R-r)^2}|y|^2 1_{|x|<r_{{k+2}}}.
\end{align*}
Thus,
\begin{align*}
I^2_{k,2}
&\le  N(2-\sigma) \bigg(\int_{B_{\Tilde{r}_k}} + \int_{B^c_{\Tilde{r}_k}} \bigg) |\zeta_k(x+y) - \zeta_k(x)- y\cdot \nabla\zeta_k(x)| |u(t,x)||y|^{-d-\sigma} \,dy\\
&\le N (2-\sigma)\frac{2^{2k}}{(R-r)^2} |u(t,x)| 1_{|x|<r_{{k+2}}} \int_{B_{\Tilde{r}_k}} |y|^{2-d-\sigma} \, dy\\
& \quad + N \int_{B^c_{\Tilde{r}_k}} \bigg(1_{|x+ y|< r_{{k+1}}}+ 1_{|x|< r_{{k+1}}}(1 + \frac{2^{k}}{(R-r)}|y|)\bigg) |u(t,x)| |y|^{-d-\sigma}  \, dy,
\end{align*}
which, similar to \eqref{tt1}, \eqref{tt2}, and \eqref{tt3}, implies that
\begin{align*}
\norm{I^2_{k,2}}_{L_p(\rr^d_T)} &\le  N \frac{2^{\sigma k}}{(R-r)^{\sigma}} \norm{u}_{p,R}\\
&\quad {+N \frac{2^{(d+\sigma)k}}{(R-r)^{d+\sigma}}R^{d/p}(1 + R^{d+\sigma} ) \norm{u}_{L_p((0,T);L_1(\rr^d,\psi))}}.
\stepcounter{equation}\tag{\theequation}\label{ee2}
\end{align*}
By \eqref{ee1} and \eqref{ee2},
\begin{align*}
I_k &\le \norm{I^2_{k,1}}_{L_p(\rr^d_T)} + \norm{I^2_{k,2}}_{L_p(\rr^d_T)} \\
&\le  N \frac{2^{(\sigma-1) k}}{(R-r)^{\sigma-1}} \norm{Du}_{p,r_{k+3}} + N  \frac{2^{\sigma k}}{(R-r)^{\sigma}} \norm{u}_{p,R}  \\
&\quad+ N \frac{2^{(d+\sigma)k}}{(R-r)^{d+\sigma}}R^{d/p}(1 + R^{d+\sigma} ) \norm{u}_{L_p((0,T);L_1(\rr^d,\psi))}.
\end{align*}

{\em Case 3: $\sigma =1$.} In this case, by \eqref{long} and \eqref{sig1}, for any $\delta_k >0$ to be determined,
\begin{align*}
&|L(\zeta_k u)(t,x) - \zeta_k Lu(t,x)|\\
&\le N \int_{B_{\delta_k}}|\zeta_k(x+y) - \zeta_k(x)||u(t,x+y) - u(t,x)||y|^{-d-1}\,dy\\
&\quad +N \int_{B_{\delta_k}}|\zeta_k(x+y) - \zeta_k(x)- y\cdot \nabla\zeta_k(x)| |u(t,x)||y|^{-d-1} \,dy\\
&\quad +N \int_{B^c_{\delta_k}}  |\zeta_k(x+y) - \zeta_k(x)||u(t,x+y) ||y|^{-d-1}\,dy\\
&=: I^3_{k,1} +  I^3_{k,2} +  I^3_{k,3}.
\end{align*}
The estimates of $I^3_{k,1}$, $I^3_{k,2}$, and  $I^3_{k,3} $ are similar to that of $I^2_{k,1}$, $I^2_{k,2}$, and $I^1_{k,2}$ respectively. It follows that
\begin{align*}
\norm{I^3_{k,1}}_{L_p(\rr^d_T)} &\le  N \frac{2^{k}}{(R-r)}\delta_k\norm{Du}_{p,r_{k+3}}, \quad
  \norm{I^3_{k,2}}_{L_p(\rr^d_T)} \le  N \frac{2^{2k}}{(R-r)^2}\delta_k \norm{u}_{p,R},
\end{align*}
and
\begin{align*}
\norm{I^3_{k,3}}_{L_p(\rr^d_T)} &\le  N \delta_k^{-1} \norm{u}_{p,R} +  N R^{d/p} \frac{1 + (R + \delta_k)^{d+1}}{\delta_k^{d+1}}  \norm{u}_{L_p((0,T);L_1(\rr^d,\psi))},
\end{align*}
where $N$ is independent of the choice of $\delta_k$. For any $\ep \in (0,1)$, let $\delta_k = \ep^3\Tilde{r}_k/N$.
Then,
\begin{align*}
I_k& \le \norm{I^3_{k,1}}_{L_p(\rr^d_T)} + \norm{I^3_{k,2}}_{L_p(\rr^d_T)} + \norm{I^3_{k,2}}_{L_p(\rr^d_T)}\\
&\le \ep^3\norm{Du}_{p,r_{k+3}} +  N\frac{2^{k}}{R-r}\ep^{-3} \norm{u}_{p,R} \\
&\quad +  N R^{d/p}\bigg(1+ \frac{2^{(d+1)k}\ep^{-3(d+1)}}{(R-r)^{d+1}}  (1 + R^{d+1} )\bigg) \norm{u}_{L_p((0,T);L_1(\rr^d,\psi))}.
\end{align*}
The lemma is proved.
\end{proof}

\begin{assumption}\label{ass44}
For given $E \subset F \subset (-\infty,T)$, if
$$\left| \mathcal{C}_R(t,x) \cap E \right| \geq \gamma |\mathcal{C}_R(t,x)|$$
for $(t,x) \in (-\infty,T] \times \rr^d$ and $R >0$, then
\[
 \widehat{\mathcal{C}}_R (t,x) \subset F,
\]
where $\mathcal{C}_R(t,x)$ and $\widehat{\mathcal{C}}_R (t,x)$ are defined in \eqref{qq}.
\end{assumption}
\begin{lemma}[Crawling of ink spots]\label{a6}
Let $\gamma \in (0,1)$ and $|E| < \infty$.
Suppose that $E \subset F \subset (-\infty,T) \times \rr^d$ satisfy Assumption \ref{ass44}.
Then
\[
|E| \leq N(d,\alpha) \gamma |F|.
\]
\end{lemma}
\begin{proof}
See \cite[Lemma A.20]{dong19}
\end{proof}

\begin{lemma}[Sobolev embeddings]\label{a5}
Let $\alpha \in (0,1]$, $\sigma\in(0,2)$, $T \in (0, \infty)$, $p  \in (1, \infty)$, and $u \in {\hh_{p,0}^{\alpha,\sigma}(T)}$.
\begin{enumerate}
    \item If $p < d/\sigma + 1/\alpha$, then there exists $q\in(p,\infty)$ satisfying
    \[1/q= 1/p -\alpha\sigma/(\alpha d + \sigma)
    \]
    such that for all $l\in[p,q]$,
    \[ \norm{u}_{L_l(\rr^d_T)} \le N  \norm{u}_{\hh_{p}^{\alpha,\sigma}(T)}, \stepcounter{equation}\tag{\theequation}\label{am11}
\]
where $N = N(d,\alpha,\sigma,p,l,T)$.
    \item If $p = d/\sigma + 1/\alpha$, then the same estimate holds for $l\in[p,\infty)$.
    \item If $p > d/\sigma + 1/\alpha$, then
    \[ \norm{u}_{L_\infty(\rr^d_T)} \le N  \norm{u}_{\hh_{p}^{\alpha,\sigma}(T)},
\]
where $N = N(d,\alpha,\sigma,p,{T})$.
\end{enumerate}
\end{lemma}
\begin{proof}
(1) By the mixed derivative theorem \cite{mi}, we have
\[
\hh_{p,0}^{\alpha,\sigma}(T) \hookrightarrow H_p^{\alpha(1-\theta)}([0,T];H_p^{\sigma\theta}(\rr^d)) \qx{for all} \theta\in[0,1].
\]
Let $\theta = \alpha d /(\alpha d + \sigma)$ and
$$1/q = 1/p - \alpha(1-\theta) = 1/p - \sigma\theta/d= 1/p -\alpha\sigma/(\alpha d + \sigma).$$ The Sobolev embeddings
$$ H_p^{\alpha(1-\theta)}([0,T]) \hookrightarrow  L_q([0,T])\qx{and}H_p^{\sigma\theta}(\rr^d)\hookrightarrow L_q(\rr^d),$$
imply that
\begin{align*}
&\norm{u}_{L_q(\rr^d_T)} \le N \big( \int_0^T \norm{u(t,\cdot)}^q_{H_p^{\sigma\theta}(\rr^d)} dt \big)^{1/q} \\
&\le N(T) \norm{u}_{H_p^{\alpha(1-\theta)}([0,T];H_p^{\sigma\theta}(\rr^d))}
\le N \norm{u}_{\hh_{p}^{\alpha,\sigma}(T)}.
\end{align*}
Since \eqref{am11} holds for $l=p$, it holds for all $l \in [p,q]$ by H\"older's inequality.

{(2)
Using the embedding
$$ H_p^{\alpha(1-\theta)}([0,T]) \hookrightarrow  L_l([0,T])\qx{and}H_p^{\sigma\theta}(\rr^d)\hookrightarrow L_l(\rr^d),$$
the proof is similar to that of (1).}

(3) Due to the facts that $\alpha(1-\theta)p > 1$ and $\sigma\theta p > d $, and the Sobolev embeddings
$$ H_p^{\alpha(1-\theta)}([0,T]) \hookrightarrow  L_\infty([0,T])\qx{and}H_p^{\sigma\theta}(\rr^d)\hookrightarrow L_\infty(\rr^d),$$ we have
\begin{align*}
&\norm{u}_{L_\infty(\rr^d_T)} \le  N(T) \norm{u}_{H_p^{\alpha(1-\theta)}([0,T];H_p^{\sigma\theta}(\rr^d))}
\le N \norm{u}_{\hh_{p}^{\alpha,\sigma}(T)}.
\end{align*}

\end{proof}

Next, we prove an embedding into H\"older spaces. Note that we only use the result when $\alpha = 1 $ for this paper.
\begin{lemma}[A H\"older estimate] \label{aa7}
Let $\alpha \in (0,1]$, $\sigma\in(0,2)$, $T \in (0, \infty)$, $p  \in (1, \infty)$, $p \in (d/\sigma + 1/\alpha,2d/\sigma + 2/\alpha)$, and $u \in {\hh_{p,0}^{\alpha,\sigma}(T)}$. There exists $\tau = \sigma -(d+\sigma/\alpha)/p \in (0,1)$ such that
    \[ \norm{u}_{C^{\tau\alpha/\sigma,\tau}(\rr^d_T)} \le N  \norm{u}_{\hh_{p}^{\alpha,\sigma}(T)},
\]
where $N = N(d,\alpha,\sigma,p,T)$.
\end{lemma}
\begin{proof}
Without loss of generality, we assume that $u \in C_0^\infty([0,T]\times \rr^d)$ with $u(0,\cdot) = 0$. Denote
\[
K = \sup \big\{\frac{|u(t_1,x) - u(t_2,y)|}{|t_1-t_2|^{\tau \alpha/\sigma} + |x-y|^\tau } : (t_1,x) ,(t_2,y) \in \rr^d_T, 0< |t_1-t_2|^{\alpha/\sigma} + |x-y|\le 1 \big\}
\]
and $$\rho= \ep (|t_1-t_2|^{\alpha/\sigma} + |x-y|), $$
where $\ep\in (0,1) $ is a constant to be determined.

Note that by Lemma \ref{a5} and the fact that $p> d/\sigma + 1/ \alpha$, it suffices to prove \[K \le N \norm{u}_{\hh_{p}^{\alpha,\sigma}(T)}. \stepcounter{equation}\tag{\theequation}\label{am8}\]
To this end, we write
\[
|u(t_1,x) - u(t_2,y)| \le |u(t_1,x) - u(t_2,x)| + |u(t_2,x) - u(t_2,y)| := J_1 + J_2.\stepcounter{equation}\tag{\theequation}\label{am3}
\]
For $J_1$, by taking $z\in B_\rho(x)$ and the triangle inequality, we have
\begin{align*}
J_1
&\le  |u(t_1,x) - u(t_1,z)| + |u(t_2,x) - u(t_2,z)| +  |u(t_1,z) - u(t_2,z)| \\
&\le 2K \rho^\tau + |u(t_1,z) - u(t_2,z)|\\
&\le 2K \rho^\tau + N(\alpha,p)|t_1-t_2|^{\alpha - 1/p}\norm{\partial_t^\alpha u(\cdot,z)}_{L_p(0,T)}, \stepcounter{equation}\tag{\theequation}\label{am1}
\end{align*}
where for the last inequality, we used \cite[Lemma A.14]{dong19} together with the fact that $\alpha > 1/p$.
Then by taking the average over $B_\rho(x)$ on both sides of \eqref{am1} together with H\"older's inequality, we get
\begin{align*}
J_1 &\le 2K \rho^\tau+ N(d,\alpha,p)|t_1-t_2|^{\alpha - 1/p}\rho^{-d/p}\norm{\partial_t^\alpha u}_{L_p(\rr^d_T)}\\
&\le 2K \rho^\tau+ N(d,\alpha,p)\ep^{\sigma/(\alpha p)-\sigma }\rho^{\tau}\norm{u}_{\hh_{p}^{\alpha,\sigma}(T)}. \stepcounter{equation}\tag{\theequation}\label{am4}
\end{align*}

On the other hand, for $s\in I:= ( t_2 - \rho^{\sigma/\alpha},t_2+ \rho^{\sigma/\alpha}) \cap (0,T)$,
\begin{align*}
J_2
&\le  |u(t_2,x) - u(s,x)| +  |u(t_2,y) - u(s,y)| +   |u(s,x) - u(s,y)| \\
&\le 2K \rho^\tau + |u(s,x) - u(s,y)|\\
&\le 2K \rho^\tau + N(d,\sigma,p)|x-y|^{\sigma - d/p}\norm{ u(s,\cdot)}_{H^\sigma_p(\rr^d)}, \stepcounter{equation}\tag{\theequation}\label{am2}
\end{align*}
where for the last inequality, we used the embedding $H^\sigma_p(\rr^d)\hookrightarrow C^{\sigma-d/p}(\rr^d)$. Then by taking the average over $I$ on both sides of \eqref{am2} together with H\"older's inequality, we obtain
\begin{align*}
J_2 &\le 2K \rho^\tau+  N(d,\sigma,p)|x-y|^{\sigma - d/p}\rho^{{-\sigma}/(\alpha p)}\norm{u}_{L_p((0,T);H^\sigma_p(\rr^d))}\\
&\le 2K \rho^\tau+ N(d,\sigma,p)\ep^{d/p-\sigma }\rho^{\tau}\norm{u}_{\hh_{p}^{\alpha,\sigma}(T)}.\stepcounter{equation}\tag{\theequation}\label{am5}
\end{align*}
Note that $N$ is independent of $T$ because we can assume that $|I|\ge \rho^{\sigma/\alpha}$ by extending $u(t,\cdot)$ to be zero for $t<0$.

Combining \eqref{am3}, \eqref{am4}, and \eqref{am5}, we have
\[
|u(t_1,x) - u(t_2,y)| \le 4K\rho^\tau + N (\ep^{ \sigma/(\alpha p)-\sigma}+\ep^{d/p-\sigma } )\rho^{\tau}\norm{u}_{\hh_{p}^{\alpha,\sigma}(T)},
\]
where $N = N (d,\alpha,\sigma,p)$, which implies that
\[
K \le 4\ep^\tau K +  N (\ep^{-d/p}+\ep^{-\sigma/(\alpha p) } )\rho^{\tau}\norm{u}_{\hh_{p}^{\alpha,\sigma}(T)}.
\]
Thus, by picking sufficiently small $\ep$ so that $4\ep^\tau<1/2$, we arrive at \eqref{am8}, and the lemma is proved.
\end{proof}

\bibliographystyle{plain}

\end{document}